\documentclass[a4paper, 12pt]{article}
\usepackage{t1enc}
\usepackage[latin1]{inputenc}
\usepackage[english]{babel}
\usepackage{amssymb}
\usepackage{amsmath}
\usepackage{graphics}
\usepackage{blkarray}
\usepackage{tikz}

\usepackage{amsfonts}
\usepackage{epsfig,multicol}

\usepackage{hyperref}

\usepackage{amsthm}

\usepackage{mathrsfs}

\newtheorem{theorem}[subsubsection]{Theorem}

\newtheorem{lemma}[subsubsection]{Lemma}
\newtheorem{corollary}[subsubsection]{Corollary}
\newtheorem{definition}[subsubsection]{Definition}

\newtheorem{formula}[subsubsection]{}

\newcommand{\arxiv}[1]{\href{http://arxiv.org/abs/#1}{\tt arXiv:\nolinkurl{#1}}}
\newcommand{\doi}[1]{\href{http://dx.doi.org/#1}{{\tt DOI:#1}}}

\newcommand{\googlebooks}[1]{(preview at \href{http://books.google.com/books?id=#1}{google books})}

\def\<{\langle}
\def\>{\rangle}

\begin{document}

\def\hpic #1 #2 {\mbox{$\begin{array}[c]{l} \epsfig{file=#1,height=#2}
\end{array}$}}
 
\def\vpic #1 #2 {\mbox{$\begin{array}[c]{l} \epsfig{file=#1,width=#2}
\end{array}$}}

\title{Existence of the $AH+2$ subfactor}
\author{Pinhas Grossman
} \maketitle 
\begin{abstract}
We give two different proofs of the existence of the $AH+2$ subfactor, which is a $3$-supertransitive self-dual subfactor with index $\frac{9+\sqrt{17}}{2} $. The first proof is a direct construction using connections on graphs and intertwiner calculus for bimodule categories. The second proof is indirect, and deduces the existence of $AH+2$ from a recent alternative construction of the Asaeda-Haagerup subfactor and fusion combinatorics of the Brauer-Picard groupoid.
  \end{abstract}
\section{Introduction}
In \cite{MR1686551} Asaeda and Haagerup constructed two ``exotic'' subfactors, which were the first examples of 
subfactors not coming from groups or quantum groups in an apparent way. One is the Haagerup subfactor, 
with index $(5+\sqrt{13})/2 $, and the other is the Asaeda-Haagerup subfactor, with index
$(5+\sqrt{17})/2 $. The former has become increasingly well understood, with two 
new constructions developed
\cite{MR1832764, MR2679382}. Indeed, it is argued in  \cite{MR2837122}, based on analysis
of the Drinfeld center,
that the Haagerup subfactor should not be viewed as exotic at all, but rather as part of 
a conjectural infinite series of subfactors associated to finite cyclic groups of odd order 
(in which the Haagerup subfactor corresponds to $\mathbb{Z}/3\mathbb{Z} $).

The Asaeda-Haagerup subfactor (henceforth referred to as $AH$) has until recently appeared more opaque. In \cite{AHCat2}, 
a new construction of $AH$ was given by first constructing a new subfactor, which we call $2AH$, with index  twice that of $AH$. The subfactor $2AH$ is associated
to the group $\mathbb{Z}/4\mathbb{Z} \times \mathbb{Z}/2\mathbb{Z} $
in an analagous, though more complicated, manner as the Haagerup subfactor is associated to
$\mathbb{Z}/3\mathbb{Z}$. The existence of $AH$ is then deduced as a consequence
of the existence of $2AH$. This construction allowed for the solution of several
open problems regarding $AH$, notably the description of its Drinfeld center. An anlysis
of the modular data of $AH$ suggests a possible series of subfactors associated
to the groups $\mathbb{Z}/4n \mathbb{Z} \times \mathbb{Z}/2\mathbb{Z} $, of which $2AH$
is the first member (see \cite{1501.07679}).

The motivation for constructing $2AH$ came from an analysis of the Brauer-Picard groupoid
of $AH$, undertaken in \cite{GSbp}. The Brauer-Picard 
groupoid consists of all of the fusion categories in the Morita equivalence class of the even 
parts of $AH$ and all Morita equivalences between them. The input of the analysis
was the subfactor $AH$ along with two additional small-index subfactors, called $AH+1$ and $AH+2$ (with indices
$1$ and $2$ larger that that of $AH$, respectively) whose even parts belong to the 
same Morita equivalence class as those of $AH$. Starting with these three subfactors, which each 
give a Morita equivalence between two fusion categories, the groupoid was built up using essentially combinatorial 
methods. In the end, a gap in the very intricate groupoid structure pointed to the probable
existence of $2AH$, and led to the new construction of $AH$ and the results of \cite{AHCat2}.

The existence of $AH+1$ and $AH+2$ was in turn motivated by the study of quadrilaterals 
of subfactors \cite{MR1338739,MR2257402,MR2418197}. A quadrilateral of subfactors
is a square of subfactor inclusions $\begin{array}{ccc}
P&\subset &M \cr
\cup& &\cup \cr
N&\subset &Q
\end{array}$ such that $P$ and $Q$ generate $M$ and intersect in $N$. In \cite{MR2418197}, a quadrilateral was constructed whose upper 
inclusions $P \subset M$ and $Q \subset M $ are both the Haagerup subfactor,
and whose lower inclusions $N \subset P $ and $N \subset Q $ both have index one larger; 
the Galois group of $N \subset M$ is  $\mathbb{Z}/3\mathbb{Z}$. 
Somewhat surprisingly, the principal graph of the Asaeda-Haagerup subfactor appeared naturally
in the classification of similar quadrilaterals with  Galois group $\mathbb{Z}/2\mathbb{Z}$.
This suggested that there should be a quadrilateral whose upper inclusions are each $AH$ 
and whose lower inclusions have index one larger.

The subfactor $AH+1$ was constructed in \cite{MR2812458} 
by showing the existence of a certain algebra in one of the even parts of $AH$. Verifying the 
existence of this algebra involved computing several complicated intertwiner diagrams
in the bimodule category associated to $AH$. These computations were performed
using the generalized open string bimodule formalism developed in \cite{MR1686551}. They also used some data from a complicated gauge transformation calculation that was the main step in the construction of $AH$ in \cite{MR1686551}. 

Once it had been constructed, it became clear that the $AH+1$ subfactor exhibited
similar symmetries to those of $AH$, and it was conjectured that there should 
another quadrilateral whose upper inclusions are each $AH+1$ 
and whose lower inclusions have index one larger.

In this paper, an earlier version of which appeared as an online appendix to \cite{GSbp}, we construct the $AH+2$ subfactor. The basic method is similar to the construction of $AH+1$. We construct an algebra in one of the even parts of $AH+1$ by evaluating certain intertwiner diagrams. But just as evaluating these diagrams for $AH+1$ required data from the calculation in the original construction of $AH$, to construct $AH+2$ we first need to perform an analogue of Asaeda and Haagerup's calculation for the $AH+1$ subfactor. 

This calculation took up about 25 pages in \cite{MR1686551}, and the version we need is more difficult since $AH+1$ is more complicated than $AH$. We spare the reader most of the gory details, but include gauge transformation matrices in an appendix. The correctness of the gauge transformation data is verified in an accompanying Mathematica notebook. One subtlety which appears in the $AH+1$ case but did not appear in the $AH$ case is a nontrivial sign occuring in the connection of a certain period two automorphism.

We also include a second, completely different, proof of the existence of both $AH+1$ and $AH+2$. This proof is indirect and uses only the existence of $2AH $ and $AH$, the outer automorphisms of the principal even part of $2AH$, and fusion combinatorics of the Brauer-Picard groupoid. The existence of $AH$ was already deduced from the existence of $2AH$ in \cite{AHCat2}. That proof used a recognition theorem from \cite{GSbp}, in which a $4$-supertransitive subfactor can be shown to exist simply by finding a fusion category with the same fusion rules as its even part. This approach does not work for $AH+1$ and $AH+2$, since these subfactors are only $3$-supertransitive.

However, the presence of outer automorphisms of the principal even part of $2AH$ implies that the Brauer-Picard group of $AH$ has a rich structure, and the existence of $AH+1$ and $AH+2$ can be deduced using similar combinatorial methods to those in \cite{GSbp}. The success of these methods in constructing first $AH$, and now $AH+1$ and $AH+2$ as well, without any connection or intertwiner calculations at all, simply from the existence of $2AH$ and its outer automorphisms, is a reflection of the remarkable combinatorial structure of the Brauer-Picard groupoid. 

The subfactor $AH+2$ which we construct here has a number of pleasant properties. It is $3$-supertransitive, self-dual, and the odd and even part together form a $\mathbb{Z}/2\mathbb{Z} $-graded fusion category \cite{MR3354332}. There is an irreducible noncommuting quadrilateral of subfactors whose upper sides are both $AH+1$ and whose lower sides are both $AH+2$.

The paper is organized as follows.

In Section 2 we review some preliminary notions regarding subfactors, fusion categories, connections, diagrammatic calculus, and the Brauer-Picard groupoid.

In Section 3 we review some facts about the $AH$ and $AH+1$ subfactors and their constructions.

In Section 4 we construct the $AH+2$ subfactor by showing the existence of a certain algebra in one of the even parts of the $AH+1$ subfactor.

In Section 5 we give an alternative proof of the existence of both $AH+1$ and $AH+2$ from the existence of $2AH$ and combinatorics of the Brauer-Picard groupoid.

In Appendix A we give the data of a certain gauge transformation between bimodules in the bimodule category associated to $AH+1$; this data is used in Section 3 to check diagrammatic algebra relations in establishing the existence of $AH+2$.

There are two supplementary files included in the arXiv submission of this paper. The Mathematica notebook 
\textit{ahp2\_gauge.nb} 
verifies some connection calculations from Section 4 and the correctness of the gauge transformation data given in Appendix A. This Mathematica notebook is also in the arXiv submission of \cite{GSbp}. The text file 
\textit{AH1-AH4\_Bimodules} 
lists the fusion bimodules between the fusion rings $AH_1$ and $AH_4$, which are the Grothendieck rings of even parts of the subfactors $AH$ and $2AH$, respectively; it also gives their multiplicative compatiblity. This complements the text files in the arXiv submission of \cite{GSbp} which give analogous data for the $AH_i -AH_j$ fusion bimodules for $1 \leq i,j \leq 3 $. The $AH_1 -AH_4$ bimodules are used in Section 5.

\textbf{Acknowledgements.}
This paper grew out of an online appendix to \cite{GSbp}, which was joint work with Noah Snyder. In particular, the idea for the second proof of existence of $AH+1$ and $AH+2$ arose in conversations with Noah Snyder and uses the methods of \cite{GSbp} and the results of \cite{AHCat2}. We would like to thank Marta Asaeda for help in computing the connection on $AH+1$. We would like to thank Scott Morrison for initially pointing out to us that the dual graph of $AH+2$ must be the same as the principal graph. This work was partially supported by ARC grant DP140100732.

\section{Preliminaries}

\subsection{Subfactors and tensor categories}
A subfactor is a unital inclusion $N \subseteq M  $ of II$_1 $ factors. The subfactor has finite-index if the commutant $N'$ in the standard representation of $N \subseteq M$ on $L^2(M)$ is a finite von Neumann algebra, and the index is then defined as the Murray-von Neumann coupling constant of $N$ in this representation \cite{MR696688}.

The principal even part $\mathcal{N} $ of a finite-index subfactor $N \subseteq M $ is the category of $N-N$ bimodules tensor generated by $${}_N M {}_N \cong {}_N M {}_M \otimes_M {}_M M {}_N $$ and the dual even part is the category of $M-M$ bimodules tensor generated by ${}_M M {}_N \otimes_N {}_N M {}_M$. The categories $\mathcal{N} $ and $\mathcal{M} $ are C$^*$-tensor categories. The subfactor $ N \subseteq M$ is said to have finite depth if $\mathcal{N} $ and $\mathcal{M} $ have finitely many simple objects, up to isomorphism; in this case they are fusion categories.
The odd part of the subfactor is the category $\mathcal{K} $ of $N-M $ bimodules which is generated by tensoring objects of $\mathcal{N}$ with ${}_N M {}_M $; $\mathcal{K} $ is an $\mathcal{N}-\mathcal{M}$ bimodule category. Together, $\mathcal{N} $, $\mathcal{K} $, and $\mathcal{M} $ form a $2$-category whose $1$-morphisms have duals.

The principal graph of a finite-index subfactor $N \subseteq M $ is the bipartite graph with even vertices indexed by simple objects of $\mathcal{N} $ and odd vertices indexed by simple objects in $ \mathcal{K}$, with the number of edges between an even vertex ${}_N X {}_N $ and an odd vertex ${}_N Y {}_M $ given by $$dim(\text{Hom} ({}_N X {}_N \otimes_N {}_N M {}_M,{}_N Y {}_M) ) .$$ The dual graph is defined analogously, using $\mathcal{M} $ instead of $\mathcal{N} $. For a finite depth subfactor, the norm of the principal graph is the square root of the index.

\begin{definition}
An algebra in a monoidal category is an object $A$ together with maps $1 \rightarrow A $ (unit) and $A \otimes A \rightarrow A$ (multiplication) satisfying the usual associativity and identity relations. An algebra in a C$^*$-tensor category is called a Q-system if the unit is a scalar multiple of an isometry and multiplication is a scalar multiple of a co-isometry. A Q-system $A$ is said to be irreducible if $dim(\text{Hom}(1,A))=1 $.
\end{definition}

If $N \subseteq M$ is a finite-index subfactor, then ${}_N M {}_N $ has the structure of a Q-system in $\mathcal{N} $. Conversely, given an irreducible Q-system $A$ in a C$^*$-tensor category with simple identity object, there is a finite-index subfactor $N \subseteq M $ whose prinicipal even part $\mathcal{N} $ is equivalent to the tensor category generated by $A$ \cite{MR1257245}.

In a C$^*$-tensor category with simple identity object, there is a notion of dimension of objects, which is positive for nonzero objects, multiplicative in tensor products, and additive in direct sums. The dimension of an irreducble Q-system is the index of the corresponding subfactor \cite{MR1444286}. 

\subsection{Connections and bimodules}
The theory of paragroups and connections on graphs is due to Ocneanu.
A $4$-graph is a square of bipartite finite graphs
$\mathcal{G}_i, \ i \in \mathbb{Z}_4 $ on vertex sets $V_i, \ i \in \mathbb{Z}_4 $, as in Figure \ref{4g}.
A biunitary connection $\alpha $ consists of a $4$-graph and a function assigning complex numbers to cells, which are loops around the square.

\def\gb{{\beta^2}}
\def\gbf{{\beta^4}}
\def\gbs{{\beta}}
\def\os{{\rm OpString}}
\def\cho{\choose}
\def\ovl{\overline}
\def\sk{{*_{\cal K}}}
\def\sl{{*_{\cal L}}}
\def\sg{{*_{\cal G}}}
\def\lan{\langle}
\def\ran{\rangle}
\def\oX{{\ovl X}}
\def\begeq{\begin{eqnarray*}}
\def\endeq{\end{eqnarray*}}
\def\pha{\phantom}
\def\la{\lambda}
\def\lad{\lambda^2}
\def\R{R}
\def\Y{Y}
\def\Z{Z}
\def\sm{\small}
\def\s#1{$#1_{\sigma}$}
\def\ss#1{$#1_{\sigma^2}$}
\def\si#1{#1_{\sigma}}
\def\ssi#1{#1_{\sigma^2}}
\def\a{\alpha}
\def\ab{\tilde{\alpha}}
\def\dline(#1,#2){\multiput(#1,#2)(0,-3){8}{\line(0,-1){2}}}
\def\hl{\hline}
\def\mcol{\multicolumn}
\def\td#1{\tilde #1}
\def\bl{$\bullet$}
\begin{figure}[h] 
\begin{center}
\thinlines
\unitlength 1.0mm
\begin{picture}(30,20)(0,-5)
\multiput(11,-2)(0,10){2}{\line(1,0){8}}
\multiput(10,7)(10,0){2}{\line(0,-1){8}}
\multiput(10,-2)(10,0){2}{\circle*{1}}
\multiput(10,8)(10,0){2}{\circle*{1}}
\put(6,12){\makebox(0,0){$V_0$}}
\put(24,12){\makebox(0,0){$V_1$}}
\put(6,-6){\makebox(0,0){$V_3$}}
\put(24,-6){\makebox(0,0){$V_2$}}
\put(15,12){\makebox(0,0){${\cal G}_0$}}
\put(15,-6){\makebox(0,0){${\cal G}_2$}}
\put(6,3){\makebox(0,0){${\cal G}_3$}}
\put(24,3){\makebox(0,0){${\cal G}_1$}}
\put(15,3){\makebox(0,0){$\a$}}
\end{picture}
\end{center}
\caption{Schematic representation of a connection; cells are loops around the square.  }
\label{4g}
\end{figure}
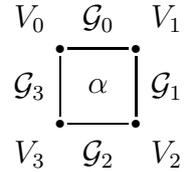

This function is required to satisfy several axioms: unitarity, initialization, harmonicity, and renormalization; see \cite{MR1642584}.

For   a finite depth subfactor $N \subseteq M $, one considers the following $4$-graph: the upper left vertices $V_0$ are the even vertices of the principal graph, the lower right vertices $V_2$ are the even vertices of the dual graph, and $V_1$ and $V_3$ are each the (common) odd vertices of the principal and dual graphs. The upper graph $\mathcal{G}_0 $ and the left graph $\mathcal{G}_3 $ are each the principal graph, with the even vertices of $\mathcal{G}_0 $ identified with the duals of the corresponding vertices of $\mathcal{G}_3 $, and the lower and right graphs $\mathcal{G}_1 $ and $\mathcal{G}_2 $ are each the dual graph, again with the even vertices identified according to duality. Then there is a biunitary connection on this $4$-graph associated to $N \subseteq M $ whose gauge equivalence class is a complete invariant for the subfactor.

In fact one can construct a subfactor from any biunitary connection on a $4$-graph (with $\mathcal{G}_0 $ and $\mathcal{G}_2 $ connected), but in general the connection of the resulting subfactor is different than the input connection. For a connection to come from a subfactor, an additional axiom called flatness is required to be satisfied. To construct a subfactor with a given pair of principal and dual graphs, one can try to write down a biunitary connection for the graphs and check for flatness. However verifying flatness is usually exceedingly difficult in practice, and Asaeda and Haagerup took a different approach to construct their subfactors.

We briefly summarize their theory of \textit{generalized open string bimodules}; for more details see \cite{MR1686551}.

Given a biunitary connection with $\mathcal{G}_0 $ and $\mathcal{G}_2 $ connected, one can associate II$_1 $ factors $N $ and $M$ to $\mathcal{G}_0 $ and $\mathcal{G}_2 $, respectively,  and an $N-M $ bimodule to the connection. There is a notion of direct sum of connections with the same horizontal graphs  $\mathcal{G}_0 $ and $\mathcal{G}_2 $, in which one takes disjoint unions of the vertical graphs. There is also a notion of product of  connections where the lower graph of the first connection is the same as the upper graph of the second connection, in which vertical edges are composed and the connection values multiplied accordingly. Finally there is an opposite connection with the upper and lower graphs reversed. These operations on connections correspond to the direct sum, relative tensor product, and contragedient operations on the corresponding bimodules over II$_1$ factors. Isomorphisms between bimodules correspond
to gauge transformations of the vertical graphs of the corresponding connections. We will often identify connections with their corresponding bimodules.

If $N \subseteq M $ is a finite depth hyperfinite subfactor with connection $\kappa $, then the bimodule ${}_{N^{\#}} X {}_{M^{\#}} $ associated to $\kappa $ gives a subfactor $N^{\#} \subseteq (M^{\#})' $ which is isomorphic to $ N \subseteq M$. Then by taking products of $\kappa$ and its opposite connection, and decomposing these products into irreducible summands, we get a concrete representation of the $2$-category of bimodules associated to $ N \subseteq M$, which allows us to perform calculations involving intertwiners.

\subsection{Diagrammatic calculus}

To perform calculations in the $2$-category coming from a subfactor, we use a standard diagrammatic calculus. Intertwiners are represented by vertices or boxes, with emanating edges labeled by the source and target objects. Following sector notation, we use Greek letters to label objects and often suppress tensor product symbols and ``Hom''. 
Thus for example the diagram

$$\hpic{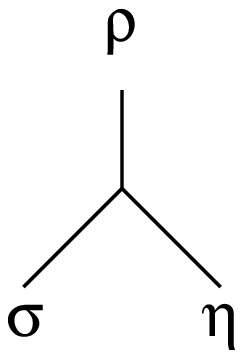} {0.9in}  $$ 

represents an intertwiner in $(\rho,\sigma\eta) $.

 Composition is represented by vertical concatenation of diagrams, and tensor product is represented by horizontal concatenation of diagrams. Straight strings represent identity morphisms, and diagrams are read from top to bottom. 
 
 If $\rho$ and $\bar{\rho} $ are contragredient bimodules, then there are scalar multiples of isometries  
  $$ \hpic{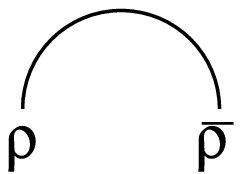} {0.5in}  \in (1,\rho \bar{\rho}) , \quad \hpic{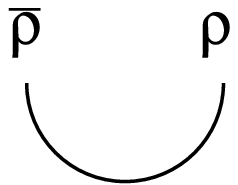} {0.6in}  \in (\bar{\rho} \rho,1) ,$$ called coevaluation and evaluation,
such that $$\hpic{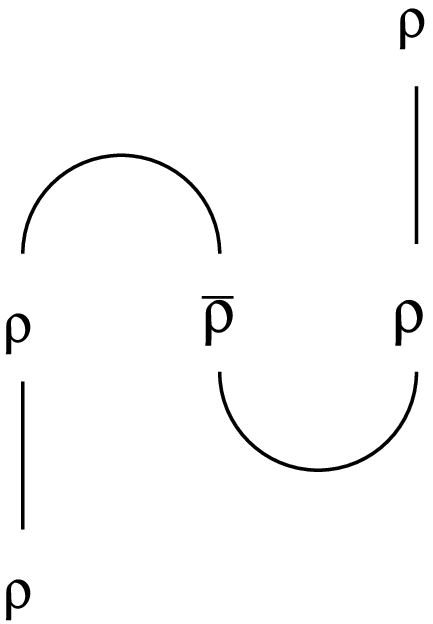} {1.2in}  =  \hpic{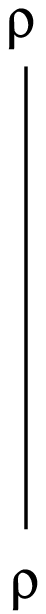} {1.2in}, \text{ and }  \quad \hpic{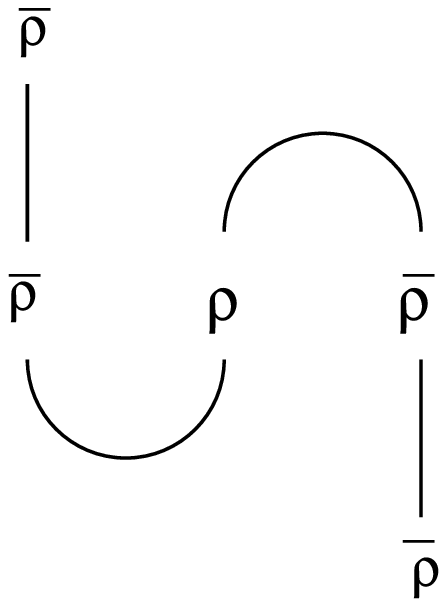} {1.2in}  =  \hpic{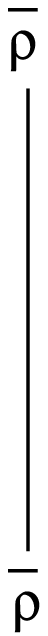} {1.2in}   .$$
A self-dual bimodule $\rho \cong \bar{\rho} $ is called symmetrically self-dual (or real) if under the identification of $\rho  $ with $\bar{\rho} $, the 
evaluation and coevaluation maps are adjoints of each other.

We now consider the $2$-category coming from a biunitary connection on a $4$-graph with connected horizontal graphs as in Figure \ref{4g}. For a pair of connections with the same horizontal graphs, bimodule intertwiners are described by collections of maps on the vertical edge spaces. That is, if $\rho $ and $\sigma$ are two connections with the same horizontal graphs $\mathcal{G}_0 $ and $\mathcal{G}_2 $ and $u \in (\rho,\sigma )$  is an intertwiner, than for each pair of vertices $a \in V_0 $ and $b \in V_3 $, we have a map from the vector space with basis indexed by the edges connecting $a$ and $b$ in the left graph of $\rho $ to the vector space with edges indexed by the edges connecting $a$ and $b$ in the left graph of $\sigma $, and similarly for each pair of vertices $c \in V_1$ and $d \in V_3$ (for the right graphs). The collection of these linear maps for all pairs of vertices in $(V_0,V_3) $ and $(V_1,V_2) $  completely determines $u $, and composition of intertwiners is given by composition of the corresponding linear maps on the vertical edge spaces.

If $u \in (\rho,\sigma) $ is an intertwiner, $a \in V_0$ and $b \in V_3$ are vertices, and
$(ab)_i $ and $(ab)_j$ are edges connecting $a$ and $b$ in the left graphs of $\rho $ and $\sigma $, respectively, then we denote by $u((ab)_i,(ab)_j) $ the corresponding coefficient of the vertical edge space map associated to $u$. We can represent coefficients of intertwiners between tensor products
of bimodules by coloring the regions of the intertwiner diagrams with vertices of the $4$-graph and the strings of the diagram with edges (except that in all of the diagrams in this paper, there is a unique edge connecting each pair of vertices, so we omit the labeling of the edges).

 Thus for example the diagram
 
 $$\hpic{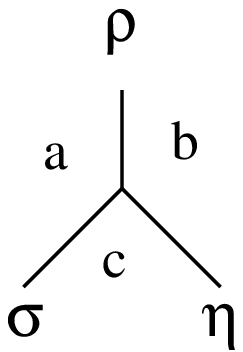} {0.9in}  $$ 
signifies the value of the coefficient of the intertwiner for the edge connecting $a$ and $b$ in a vertical graph of $\rho $ and the product of the edge connecting $a $ and $c$ in a vertical graph of $\sigma $ with the edge connecting $ c$ and $b$ in a vertical graph of $\eta $. To evalaute coefficients of more complicated intertwiner diagrams, we start by labeling the top and bottom of the diagram with the edges of the coefficient we want, and then we must sum over all \textit{states}, which are ways of filling in the diagram with consistent labeling. Each state is in turn evaluated by taking the products of the values that the labeling assigns to each vertex in the diagram. 

A key point is that in most of the computations below, the intertwiner we are looking at lives in a $1$-dimensional space, and is thus uniquely specified by a single nonzero coefficient. Thus we can identify relatively complicated intertwiners which are built out of numerous compositions and tensor products of smaller intertwiners simply by labeling the diagram by an appropriate state and evaluating the vertex coefficients of the diagram determined by that state. For examples of how this works, see \cite{MR2812458} or Lemma \ref{coflem} and Theorem \ref{mnthm} below.

\subsection{The Brauer-Picard groupoid}

To any finite depth subfactor $N \subseteq M $ we have associated a pair of fusion catgories $\mathcal{N} $ and $\mathcal{M} $ and a bimodule category ${}_{\mathcal{N}} \mathcal{K} {}_{\mathcal{M}}$ between them. The category ${}_{\mathcal{N}} \mathcal{K} {}_{\mathcal{M}}$ is invertible in the sense that  $${}_{\mathcal{N}} \mathcal{K} {}_{\mathcal{M}} \boxtimes_ {\mathcal{M}} {}_{\mathcal{M}} \mathcal{K}^{op} {}_{\mathcal{N}} \cong {}_{\mathcal{N}} \mathcal{N} {}_{\mathcal{N}}  ,$$
where $\mathcal{K} ^{op}$ is the opposite bimodule category, ${}_{\mathcal{N}} \mathcal{N} {}_{\mathcal{N}}$ is the trivial module category, and $\boxtimes_ {\mathcal{M}}$ is the relative tensor product of bimodule categories; and a similar identity holds for the product in the other order. An invertible bimodule category is also called a Morita equivalence.
\begin{definition} \cite{MR2677836}
The Brauer-Picard groupoid of a fusion category $ \mathcal{C}$ is the $3$-groupoid whose objects are fusion categories Morita equivalent to $\mathcal{C} $, whose $1$-morphisms are invertible bimodule categories between such fusion categories, whose $2$-morphisms are equivalences of such bimodule categories, and whose $3$ morphisms are isomorphisms of such equivalences. The Brauer-Picard group of $\mathcal{C} $ is the group of Morita autoequivalences of $\mathcal{C} $ modulo equivalence. 
\end{definition}
The Brauer-Picard group is an invariant of the Morita equivalence class, and contains as a subgroup the group of outer automorphisms of $\mathcal{C} $, which give bimodule categories by twisting the trivial bimodule category on one side by automorphisms.

An effective technique for performing calculations in the Brauer-Picard groupoid of a ``small'' fusion category using decategorified invariants was developed in \cite{GSbp}. We first compute the Grothendieck ring for each of the known fusion categories in the groupoid, then compute lists of based modules over each of these rings, and then look at how these different modules fit together into bimodules. Finally we look at how different bimodules can be composed, in the sense of being compatible with tensor products of bimodule categories. This combinatorial data provides strong constraints on the structure of the groupoid, and sometimes allows us to develop large structures from a very small amount of initial information. We refer the reader to \cite{GSbp} for details.

\section{$AH$, $AH+1$, and $AH+2$}
\subsection{The Asaeda-Haagerup subfactor}

In \cite{MR1686551}, Asaeda and Haagerup constructed a subfactor with index $\displaystyle \frac{5+\sqrt{17}}{2} $ and the graph pair in Figure \ref{ahgraphs}.

\begin{figure}
$$\hpic{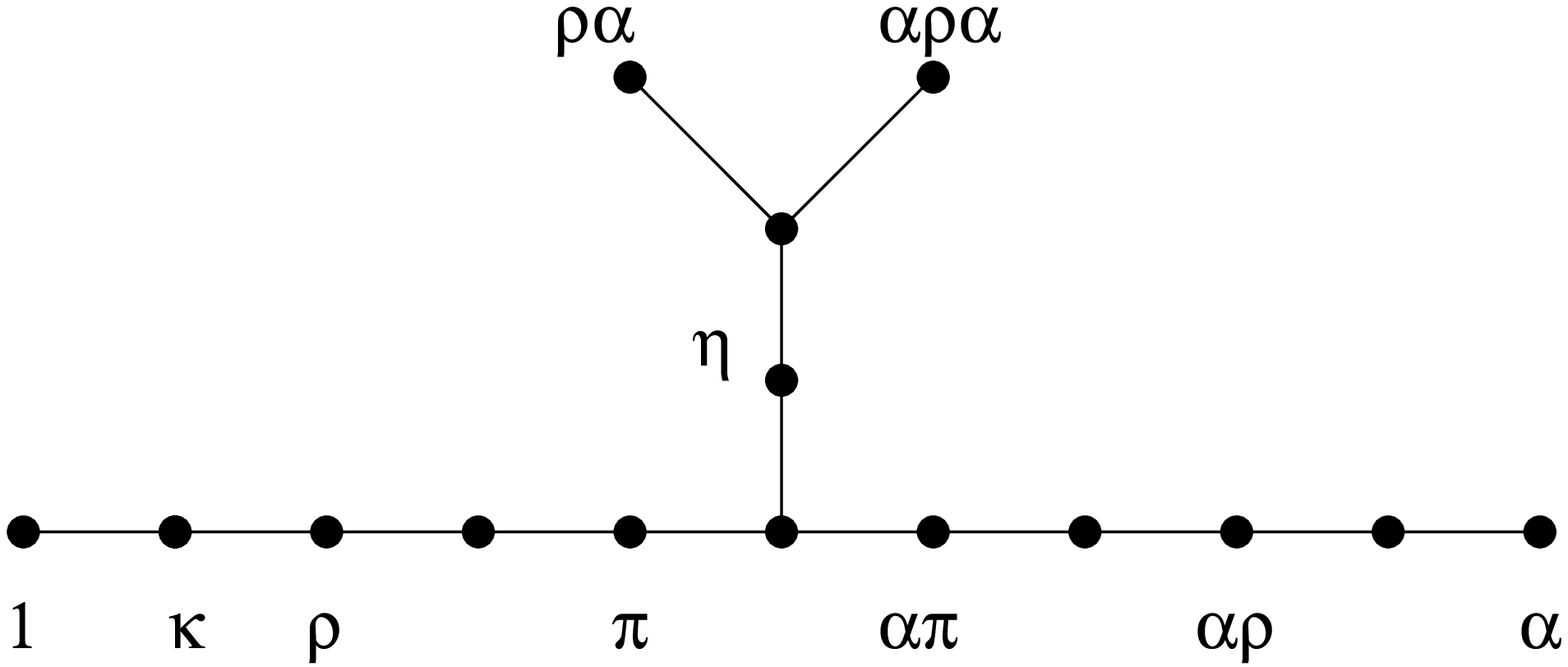} {1in} \text{ and } \hpic{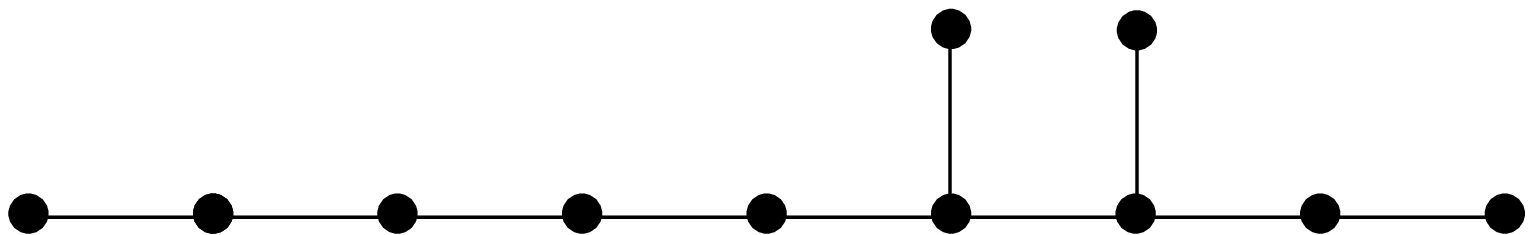} {0.3in} .$$ 
\caption{The graphs of the Asaeda-Haagerup subfactor}
\label{ahgraphs}
\end{figure}
Here we have labeled the even vertices on the principal graph, which correspond to the simple objects in the principal even part, and $\kappa$, which is the fundamental bimodule ${}_N M {}_M $. (Warning: we use different labels for the objects than in \cite{MR1686551}.)

They first computed the (unique) connection on the $4$-graph associated to this graph pair, which corresponds to the bimodule $\kappa $. Then instead of directly trying to verify flatness of this connection, they studied the $2$-category of bimodules generated by $\kappa $.

 They decomposed the product connection $\kappa \bar{\kappa} $ into a direct sum of the identity connection and another connection, which corresponds to $\rho $. 
 Note that while the upper graph of the connection $\kappa $ is 
 the principal graph and the lower graph is the dual graph, the upper and lower graphs of $\rho $ are both the principal graph, since $\rho $ is an $N-N$ bimodule.  
They then defined a connection $\alpha $ whose upper graph and lower graphs are both the principal graph, and whose vertical edges connect each vertex in the principal graph to its reflection in the vertical line through the vertex $\eta $ in Figure \ref{ahgraphs}. There is a unique connection on this $4$-graph, up to gauge equivalence, whose values are identically $1$. 
Finally, they showed that the product connections $\rho \alpha \kappa $ and $\alpha \rho \alpha \kappa $ give isomorphic bimodules. To prove this they explicitly calculated a vertical gauge transformation between these two product connections. This calculation is difficult and occupies 25 pages in their paper. From this isomorphism of bimodules, they deduced the existence of a subfactor with the given graph pair (and hence flatness of the connection on the original graph pair).

\subsection{AH+1}

In \cite{MR2812458} it was shown that with $ \kappa$ and $\alpha $ as above, there is a Q-system for $1+\bar{\kappa} \alpha \kappa $, giving a subfactor with index $1+dim(\bar{\kappa} \kappa)=\displaystyle \frac{7+\sqrt{17}}{2}$. Note that $1+\bar{\kappa} \alpha \kappa$ is an object in the dual even part of the Asaeda-Haagerup subfactor. We briefly recap the argument, since we will be using similar calculations to show existence of $AH+2$.

The following characterization of $Q$-systems for $2$-supertransitive subfactors is from
\cite{MR2418197}.
\begin{lemma} \label{plus1}
 Let $\sigma$ be a symmetrically self-dual simple object in a C$^*$-tensor category with simple unit and with $d=dim(\sigma) > 1$. Fix an isometry $\frac{1}{\sqrt{d}}\hpic{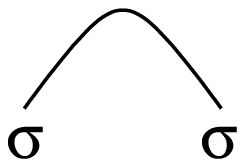} {0.3in} $.
 Then $1+ \sigma $ admits a Q-system iff there
exists an isometry $\hpic{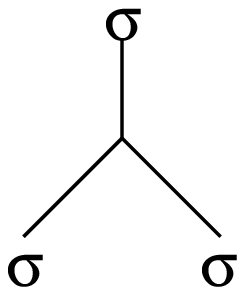} {0.5in}  $  in $(\sigma,\sigma^2) $ such
that
\begin{enumerate}
\item $$\hpic{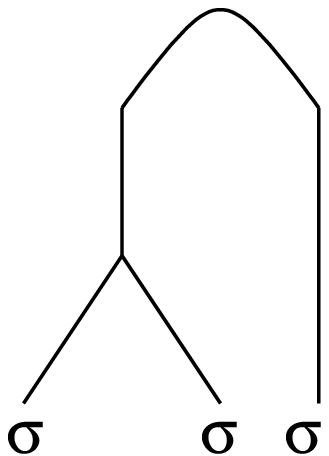} {0.8in} = \hpic{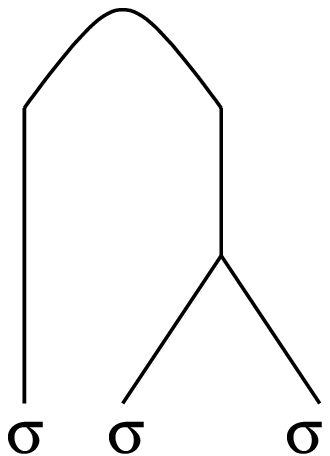} {0.8in} $$ 

\item $$\displaystyle \frac{1}{d-1} \left(\hpic{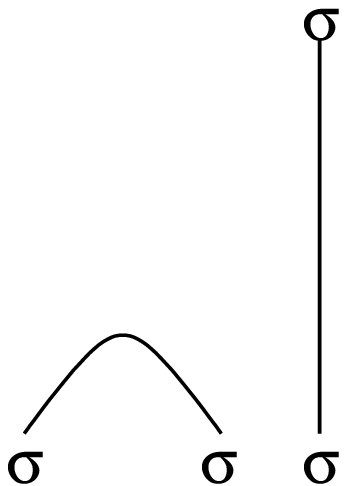} {0.8in} - \hpic{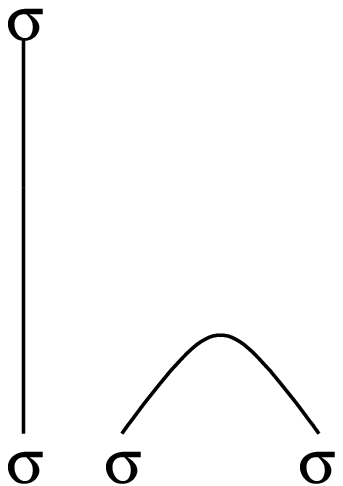} {0.8in} \right) =  \hpic{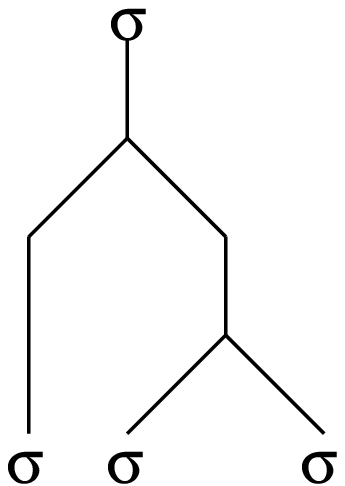} {0.8in} - \hpic{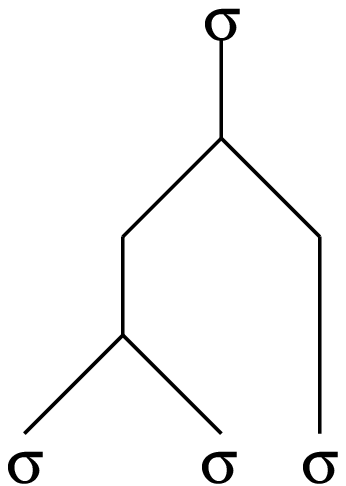} {0.8in} .$$

\end{enumerate}
\end{lemma}

For $\sigma=\bar{\kappa}\alpha\kappa $, the intertwiner space $(\sigma,\sigma^2 )$ is $1$-dimensional, and is spanned by the diagram $$ \hpic{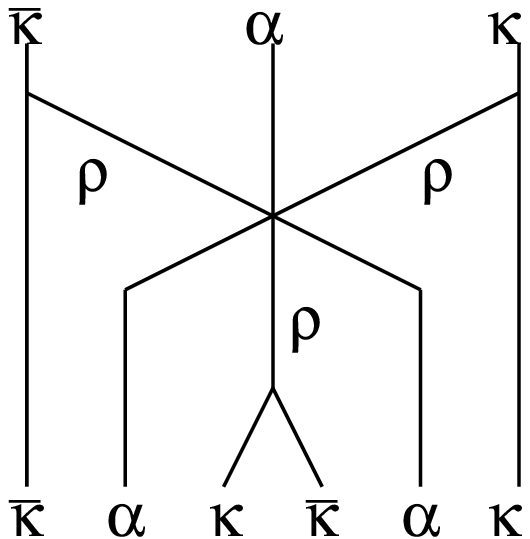} {1.0in} ,$$
where the trivalent vertices correspond to an embedding of $\rho $ in $\kappa \bar{\kappa} $ and the $6$-valent vertex corresponds to an nonzero intertwiner from $\rho\alpha \rho $ to $\alpha \rho \alpha $ (the space $(\rho\alpha \rho, \alpha \rho \alpha ) $ is also $1$-dimensional ). It is then shown that existence of the Q-system is equivalent to the following relations.

%
%
%
%

\begin{formula} \label{algebrarelations}
The Asaeda-Haagerup algebra relations:
\begin{enumerate}
\item $ \hpic{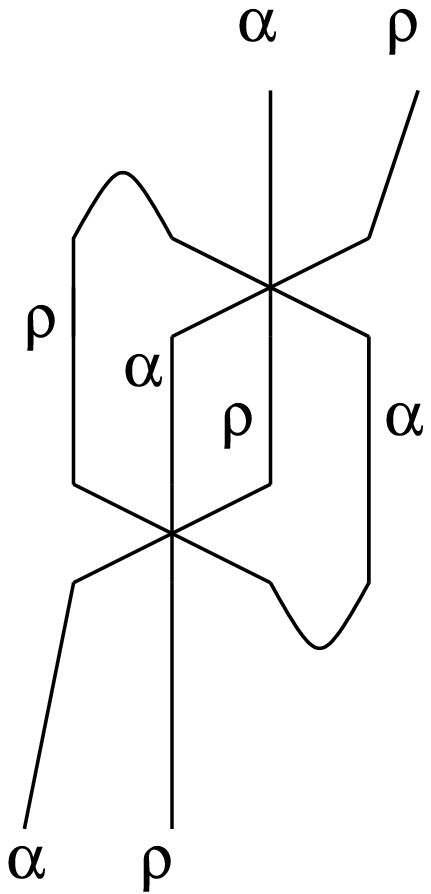} {1in} = c Id_{\alpha \rho}$, $\hpic{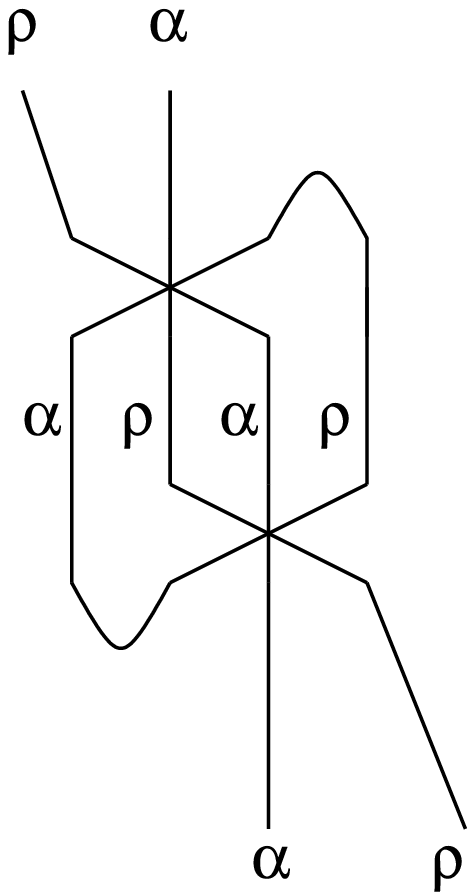} {1in} =c Id_{\rho \alpha}$

 \item $\hpic{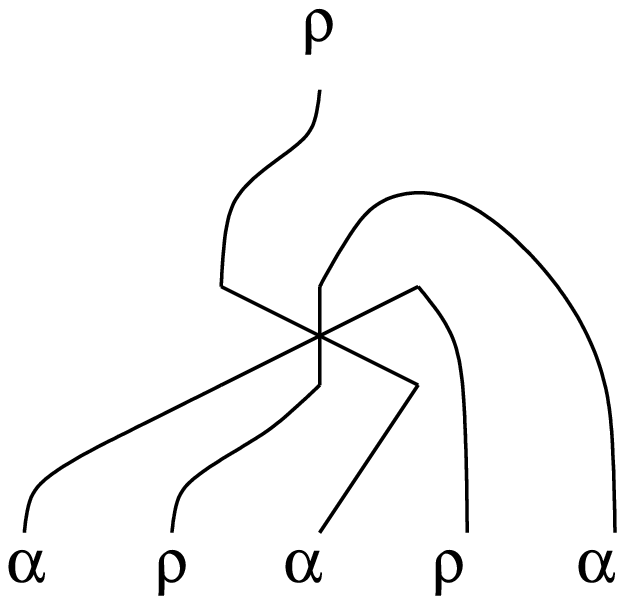} {0.88in} = \hpic{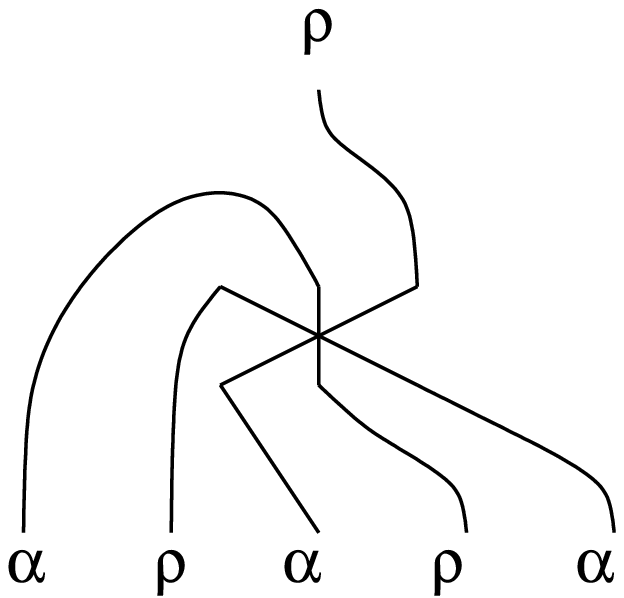} {0.88in} $

\item  $\hpic{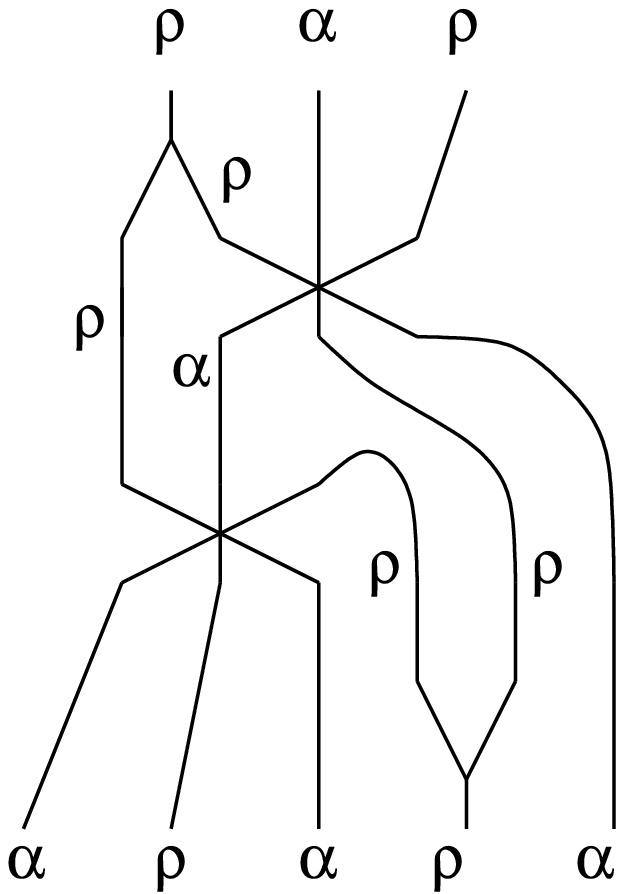} {1.2in} = \hpic{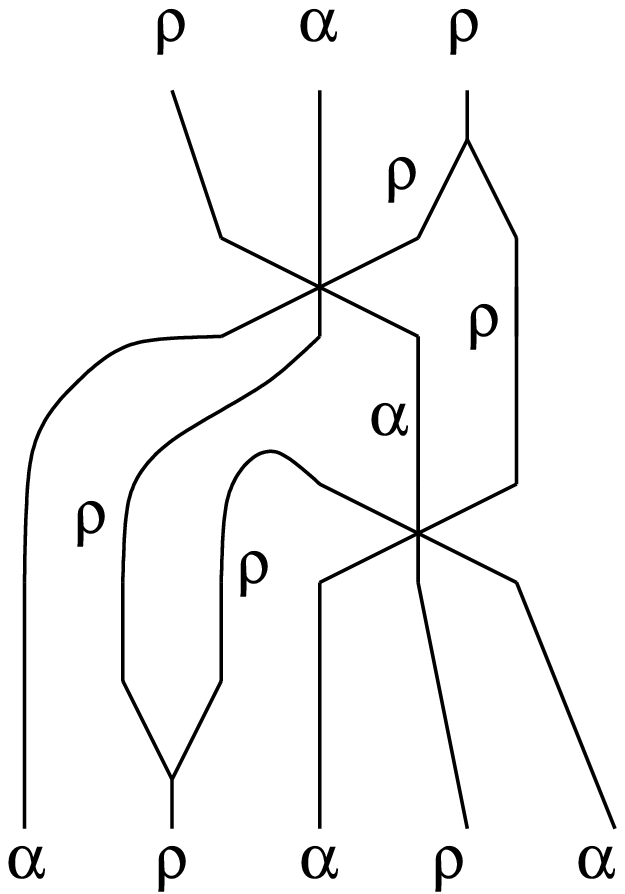} {1.2in} $

\item  $\hpic{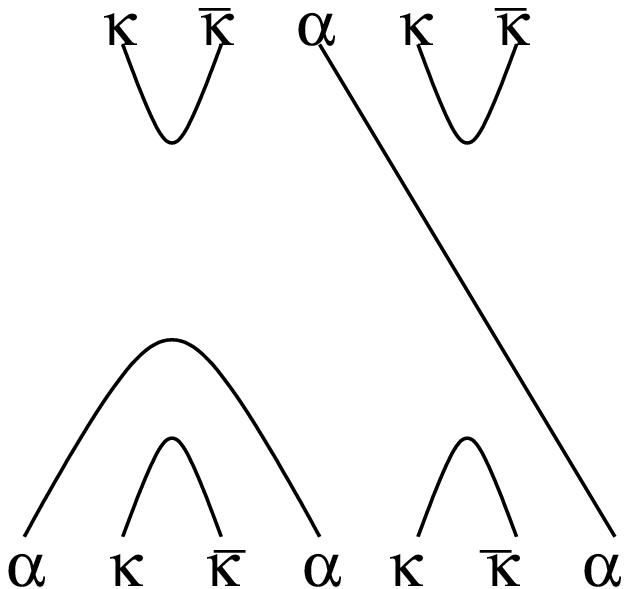} {1in} = \hpic{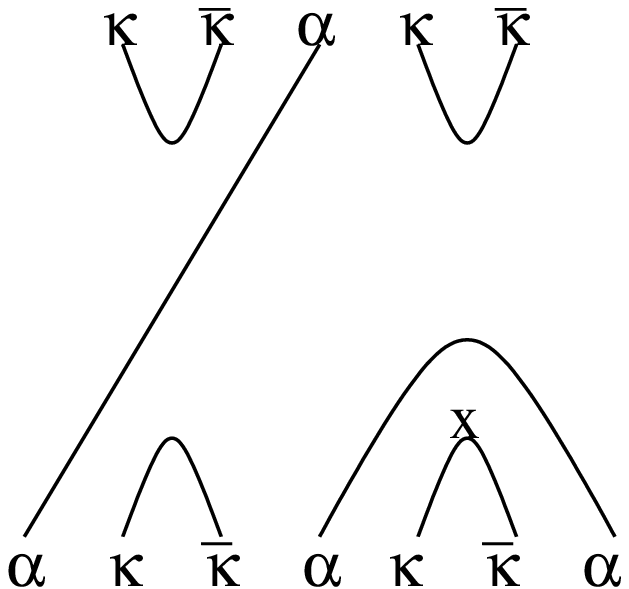} {1in} $.

\end{enumerate}
(Where $c$ is a scalar).
\end{formula}

The intertwiners in the above relations are complicated, involving many compositions and tensor products. However in all but the last equation, the intertwiners live in $1$-dimensional spaces, and are therefore determined by a single scalar coefficient. These coefficients can be found by evaluating diagrams on specifc states. The states are evaluated by decomposing the diagrams into tensor products and compositions of intertwiners $\rho \rightarrow \rho^2 $ and $\alpha \rho \alpha \rightarrow \rho \alpha \rho $, which can in turn be expressed in terms of the more elementary intertwiners   $1 \rightarrow \kappa \bar{\kappa }$,  $1 \rightarrow \bar{\kappa} \kappa$, $\rho \rightarrow \kappa \bar{ \kappa}  $, and $\rho \alpha \kappa \rightarrow \alpha \rho \alpha \kappa $. These elementary intertwiners act on vertical edges in the $4$-graphs by explicit formulas given by gauge transformation matrices.  In particular, the calculation uses data from Asaeda and Haagerup's calculation of the gauge transformation between $\rho \alpha \kappa $ and $ \alpha \rho \alpha \kappa $ to establish the Q-system relations.

\section{AH+2}
\subsection{The construction}
The graph pair for the $AH+1$ subfactor is given in Figure \ref{ahp1graphs}, where once again we have labeled the even vertices in the principal graph (recycling some of the same letters as before). 
\begin{figure} [!h]
$$\hpic{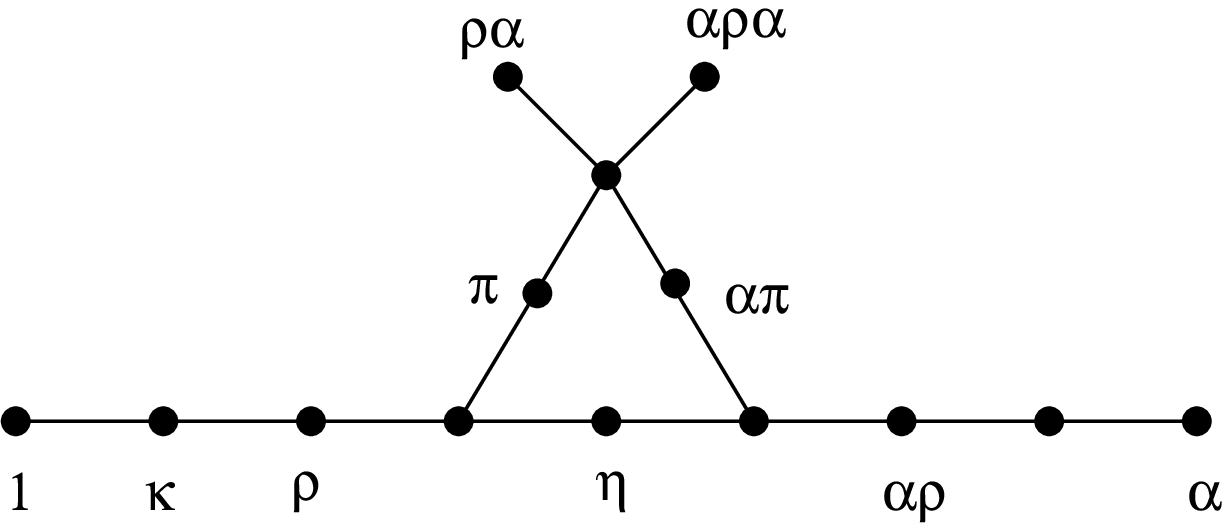} {1in} \text{ and } \hpic{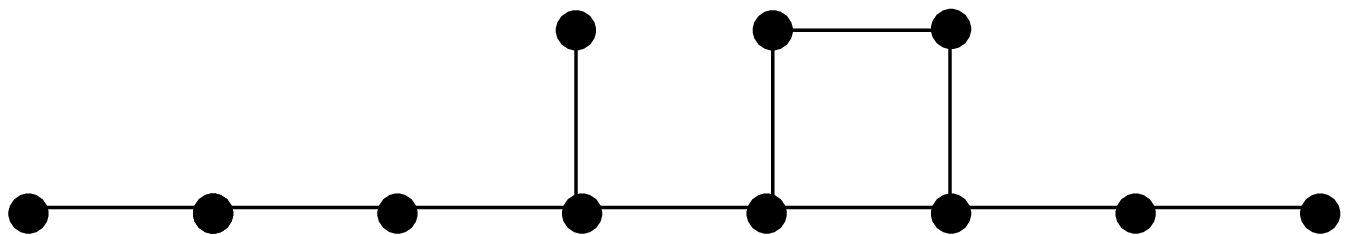} {0.3in}  .$$ 
\caption{The graphs for $AH+1$}
\label{ahp1graphs}
\end{figure}

The dual even part of $AH+1$ is the same as that of $AH$, but the principal even part is different - this can be seen by checking the Frobenius-Perron weights of the principal graph.

It was conjectured in \cite{MR2812458} that the construction of $AH+1$ can be iterated once more, and that there is again a Q-system for $1+\bar{\kappa}\alpha\kappa$, giving a subfactor with index $\displaystyle \frac{9+\sqrt{17}}{2} $. Once again the Q-system equations can be reduced to the relations $\ref{algebrarelations} $, but without a concrete realization of the $2$-category of bimodules for $AH+1 $, we have no way to evaluate the intertwiner diagrams. Therefore, we must first replicate Asaeda and Haagerup's $AH$ gauge transformation calculations for the $AH+1$ subfactor. This does not present theoretical difficulties but is somewhat more complicated than the original case.

\subsection{Connection for {AH+1}}
We are interested in the $4$-graph given in Figure \ref{4graphpics}, where we use a labeling and display similar to that used by \cite{MR1686551}. Note that in the figure we have ``unwrapped the square'', so reading from top to bottom, we have first the upper, then right, then lower, then finally left graphs. 
\begin{figure}
\label{4graphpics}
\centering
\hpic{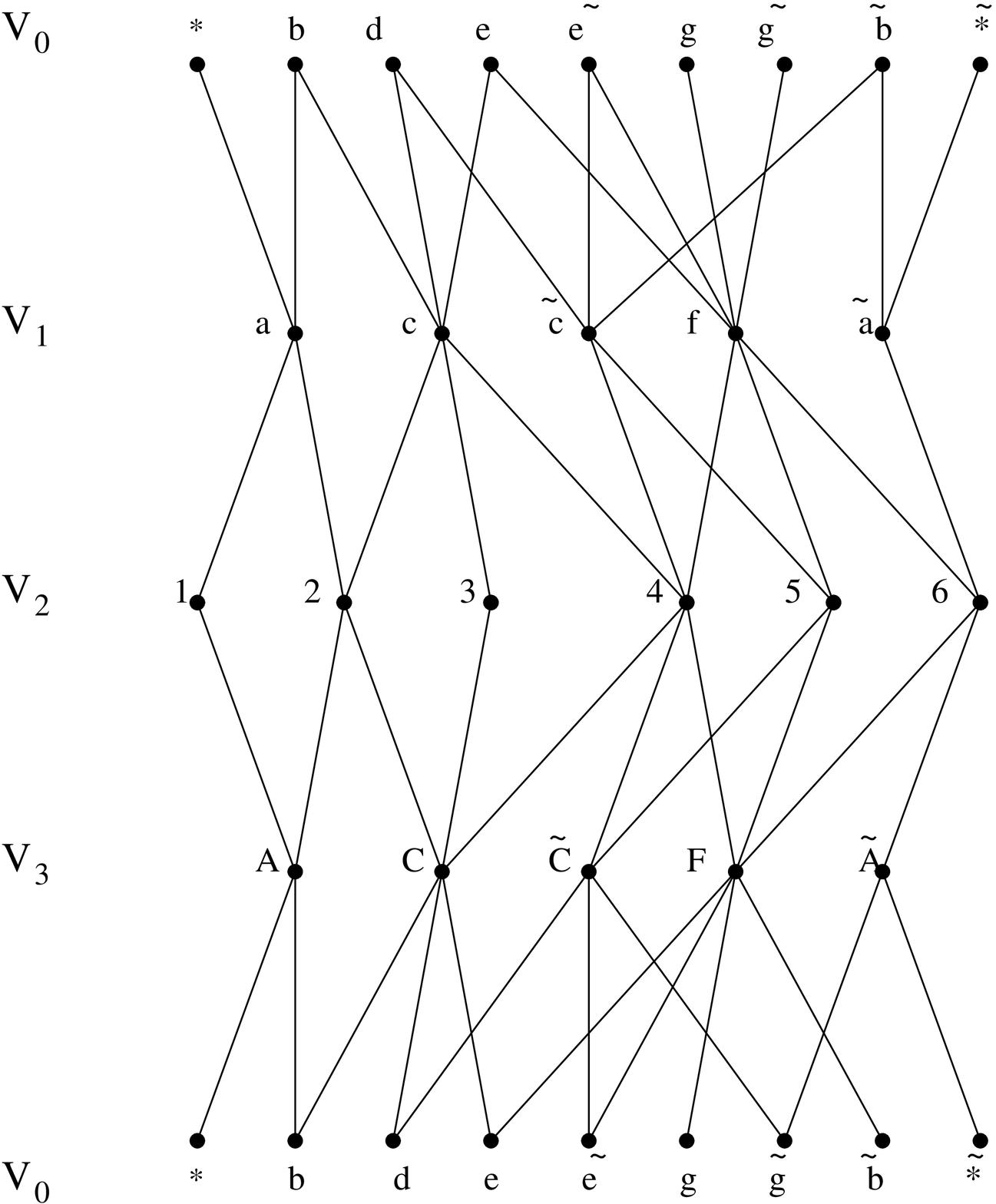} {3 in}  \quad \hpic{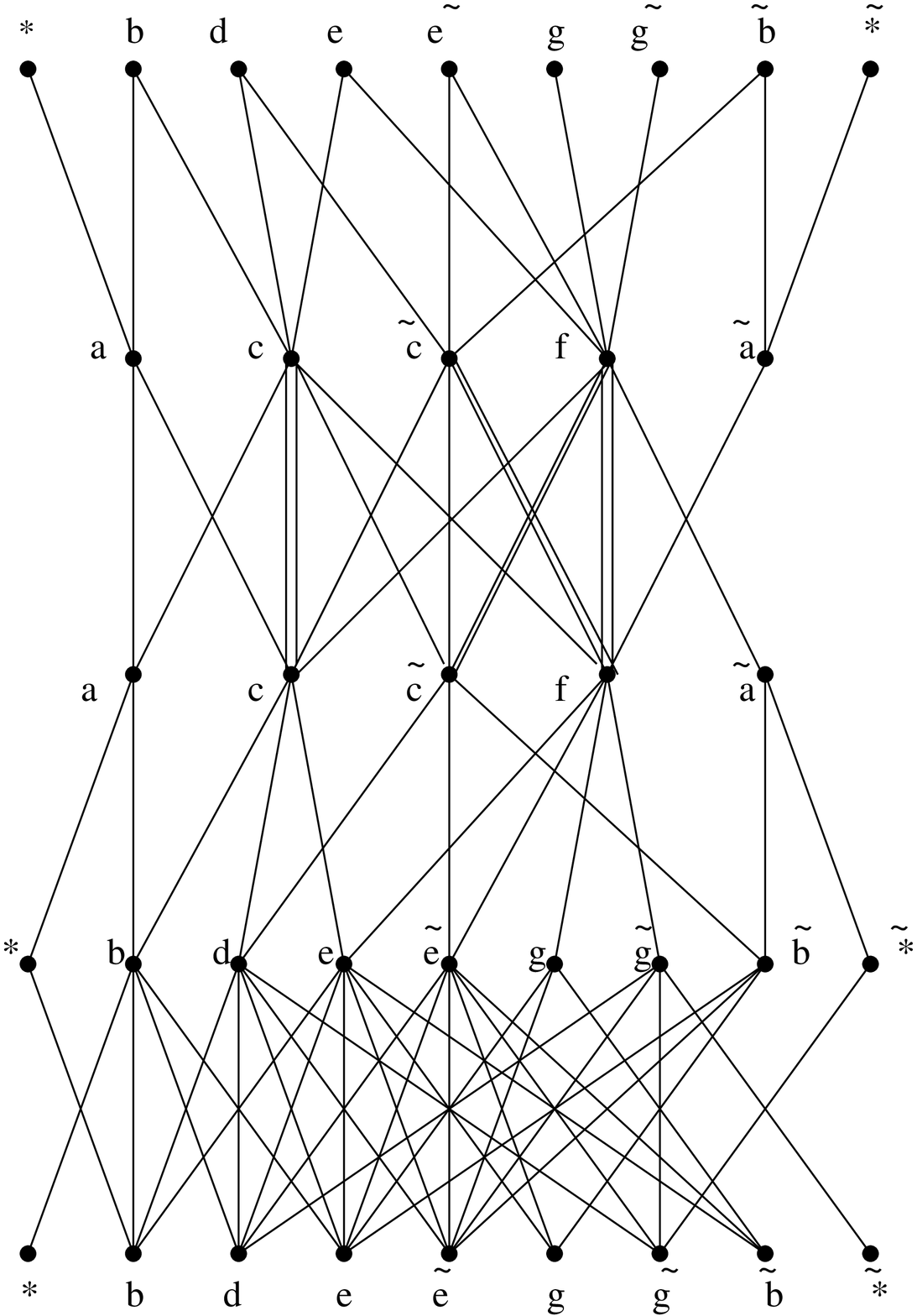} {3 in}

\caption{ The $4$-graphs for the connections of $\kappa$ (left) and $\rho $ (right) in $AH+1$}
\end{figure}

\begin{lemma}
There is a unique connection on the $4$-graph for $\kappa$ up to gauge choice, which may be taken to be real.
\end{lemma}

We now give the connection for $\kappa$ using the following notation, referring to Figure 2
for labeling of vertices. The connection is given by matrices corresponding to pairs $u-v$ with $u \in V_0 $ and $v \in V_2 $ (not to be read as ``$u$ minus $v$'') , where the rows and columns are indexed by $V_3$ and $V_1$, respectively.  

In this case the connection consists of several $2 \times 2$ matrices and a bunch of $1 \times 1$ matrices; for the $1 \times 1$ matrics we suppress the matrix notation and simply refer to the entry as $u-v$. Following the notation of \cite{MR1686551} we introduce the positive numbers $$\beta_n = \sqrt{\displaystyle \frac{7+\sqrt{17}}{2}-n } , ]quad n \leq 5 .$$ Then the connection is:

\begin{center}
$\begin{array}{c |c  c}

b-2 & a  & c  \\
  \hline
   A  & \displaystyle \frac{-1}{\beta_1^2} & \displaystyle \frac{\beta \beta_2}{\beta_1^2} \\ C & \displaystyle \frac{\beta \beta_2}{\beta_1^2} & \displaystyle \frac{1}{\beta_1^2}\\
\end{array}$ \hspace{.3in}
$\begin{array}{c |c  c}

d-4 & c  & \tilde{c}  \\
  \hline
 C  & \displaystyle \frac{-1}{\beta_{-1}} &\displaystyle\frac{\beta }{\beta_{-1}} \\

  \tilde{C} & \displaystyle\frac{-\beta}{\beta_{-1}} & \displaystyle\frac{-1}{\beta_{-1}}\\
\end{array}$ \hspace{.3in}
$\begin{array}{c |c  c}

e-4 & c  & f  \\
  \hline
 C  & \displaystyle \frac{-\beta_5}{\beta_1 \beta_3} &\displaystyle\frac{\sqrt{2} \beta}{\beta_1 \beta_3}\\

  F & \displaystyle\frac{\sqrt{2} \beta}{\beta_1 \beta_3} & \displaystyle \frac{\beta_5}{\beta_1 \beta_3}\\
\end{array}$

\vskip2ex

$\begin{array}{c |c  c}

\tilde{e}-4 & \tilde{c}  & f  \\
  \hline
 \tilde{C}  & \displaystyle \frac{2}{\beta_{-1}} &\displaystyle\frac{- \beta_{3}}{\beta_{-1}}\\

  F &\displaystyle\frac{ \beta_{3}}{\beta_{-1}} & \displaystyle \frac{2}{\beta_{-1}}\\
\end{array}$ \hspace{.5in}
$\begin{array}{c |c  c}

\tilde{e}-5 & \tilde{c}  & f  \\
  \hline
 \tilde{C}  & \displaystyle \frac{2}{\beta_1^2} & \displaystyle\frac{\sqrt{2} \beta_{-1}}{\beta_1 \beta_2}\\

  F & \displaystyle\frac{-\sqrt{2} \beta_{-1}}{\beta_1 \beta_2} & \displaystyle \frac{2}{\beta_1^2}\\
\end{array}$
\end{center}

The $1 \times 1$ entries $e-2$, $\tilde{e}-6$, and $g-5$ are $-1$; all the other $1 \times 1$ entries are $1$.

Next we want to decompose $\kappa\bar{\kappa} $ into $1+\rho$. The $4$-graph for $\rho $ can be found by removing the identity from the vertical graphs in the the product connection $\kappa \bar{\kappa} $; see Figure \ref{4graphpics}.
%
%
%
%

We first define an isometry from the identity connection to $\kappa\bar{\kappa} $ given by the vertical edge space maps in Table \ref{tabkap}.

\begin{table}
\begin{tabular}{| l | l |}
\hline
 $** \mapsto *A* $ & $aa \mapsto \frac{1}{\beta} a1a + \frac{\beta_1}{\beta} a2a$\\
 $bb \mapsto \frac{1}{\beta_1} bAb + \frac{\beta_2}{\beta_1} bCb$ & $cc \mapsto \frac{\beta_3}{\sqrt{2} \beta} c2c +
\frac{1}{\beta} c3c+\frac{\beta_{-1}}{\sqrt{2}\beta} c4c$\\
$dd \mapsto \frac{1}{\sqrt{2}} dCd + \frac{1}{\sqrt{2}} d\tilde{C}d $ & 
$\tilde{c} \tilde{c} \mapsto \frac{\beta_{-1}}{ \sqrt{2} \beta} \tilde{c}4\tilde{c} + 
\frac{\beta_1}{\sqrt{2}\beta} \tilde{c}5\tilde{c} $\\
$\tilde{*}\tilde{*} \mapsto \tilde{*}\tilde{A}\tilde{*} $& $\tilde{a}\tilde{a}\mapsto  \tilde{a}6\tilde{a}$\\
$gg \mapsto gFg$ &$ff \mapsto \frac{\beta_3}{\beta}f4f+\frac{\beta_3}{2\sqrt{2}}f5f +\frac{1}{\beta_1}f6f$\\
$ee \mapsto \frac{\sqrt{2}}{\beta_1}eCe+\frac{\sqrt{2}}{\beta_2}eFe $&\\
$\tilde{e}\tilde{e} \mapsto \frac{\sqrt{2}}{\beta_1}\tilde{e}\tilde{C}\tilde{e}+\frac{\sqrt{2}}{\beta_2}\tilde{e}F\tilde{e} $&\\
$\tilde{g}\tilde{g} \mapsto \frac{\sqrt{2}}{\beta_3}\tilde{g}\tilde{C}\tilde{g}+\frac{1}{\beta_1}\tilde{g}A\tilde{g} $&\\
$\tilde{b}\tilde{b} \mapsto \tilde{b}F\tilde{b}$&\\
\hline
\end{tabular}
\caption{The vertical edge space maps for the embedding of $1$ in $\kappa \bar{\kappa} $}
\label{tabkap}
\end{table}
To find the connection for $\rho $ we map the vertical edge spaces of its $4$-graph to the orthgonal complements of the images of the vertical edge spaces of the identity in $\kappa \bar{\kappa} $ under the map in Table \ref{tabkap}.  The $4$-graph for $\rho $ has some double edges so we use subscripts to distinguish them (e.g. $ff_1$ and $ff_2 $ are the two edges connecting $f$ to $f$ in the right vertical graph). The vertical edge space maps are given in Table \ref{tabrho}. 

\begin{table}
\begin{tabular}{|l|l|}
\hline
 $bb \mapsto \frac{\beta_2}{\beta_1} bAb -\frac{1}{\beta_1} bCb$ & $aa \mapsto \frac{\beta_1}{\beta} a1a-\frac{1}{\beta}a2a$\\
$dd \mapsto \frac{1}{\sqrt{2}} dCd - \frac{1}{\sqrt{2}} d\tilde{C}d $ & 
$cc_1 \mapsto \frac{1}{ \beta \beta_2} c2c-\frac{\beta_{-1}}{2\sqrt{2}}c3c+\frac{\beta_{-1}}{2(\beta_2)^2}c4c  $\\
$ee \mapsto \frac{\sqrt{2}}{\beta_2}eCe-\frac{\sqrt{2}}{\beta_1}eFe $&$cc_2 \mapsto \frac{\beta_3}{ \beta_2} c2c-\frac{1}{\beta_2}c4c  $\\
$\tilde{e}\tilde{e} \mapsto \frac{\sqrt{2}}{\beta_2}\tilde{e}\tilde{C}\tilde{e}-\frac{\sqrt{2}}{\beta_1}\tilde{e}F\tilde{e} $&$\tilde{c}f_1\mapsto \tilde{c}4f$\\
$\tilde{g}\tilde{g} \mapsto \frac{1}{\beta_1}\tilde{g}\tilde{C}\tilde{g}-\frac{\sqrt{2}}{\beta_3}\tilde{g}A\tilde{g} $&$\tilde{c}f_1\mapsto \tilde{c}5f $\\
&$ f\tilde{c}_1\mapsto f4 \tilde{c} $\\
& $f\tilde{c}_2\mapsto f5\tilde{c} $\\
& $ff_1 \mapsto  \frac{1}{\sqrt{2}}f4f-\frac{1}{\sqrt{\beta_1}}f5f-\frac{1}{\beta_3}f6f $\\
& $ff_2 \mapsto  -\frac{\beta_5}{2\beta_2}f4f+\frac{\beta_1}{2\sqrt{2}}f5f-\frac{1}{\beta_3}f6f $\\
\hline
\end{tabular}
\caption{The vertical edge space maps for the embedding of $\rho $ in $\kappa \bar{\kappa} $. The coefficients associated to all simple edges between distinct vertices are set to equal $1$.}
\label{tabrho}
\end{table}

The connection for $\rho $ is then defined by pulling back the connection from $\kappa \bar{\kappa} $ using this map. With this definition, we have $\kappa \bar{\kappa} =1+\rho$, as required. 

Next, we will need the connection for $\alpha $. As in the $AH$ case, the vertical graphs for $\alpha $ connect each vertex to its reflection in the vertical line through $\eta $ in Figure \ref{ahp1graphs}. However, unlike in the $AH$ case, where the only connection for the $4$-graph of $\alpha $ up to gauge equivalence is the trivial one, here there are two different possible connections for the $4 $-graph of $ \alpha$.

\begin{lemma}
The connection for $\alpha $ has all entries equal to $1$ except for the $e-f$ entry, which is $-1$.
\end{lemma}
\begin{proof}
There are two connections up to gauge equivalence: the one mentioned in the statement and the one with all entries equal to $1$. However for the connection with all entries equal to $1$, we discovered by trial and error that the connections $\rho \alpha \kappa $ and $\alpha \rho \alpha \kappa $ are not vertical gauge equivalent.
\end{proof}

Finally, we compute the composite connections $\rho \alpha \kappa $ and $\alpha \rho \alpha \kappa $ by direct mutiplication, and then compute a vertical gauge transformation between them. The gauge transformation matrices are given in Appendix \ref{bigint}, and their correctness is verified in the Mathematica notebook accompanying the arXiv submission of this paper.

\subsection{Verifying the Asaeda-Haagerup algebra relations}
Now that we have the necessary connections and gauge transformations, we are ready to evaluate intertwiner diagrams and verify the Q-system equations for $1+\bar{\kappa} \alpha \kappa $.

 First we fix some basic intertwiners. Let $r_{\kappa} \in (1,\kappa \bar{\kappa}) $ be the isometry defined by the vertical edge space maps in Table \ref{tabkap}. Let  $v \in (\rho,\kappa \bar{\kappa} )$ be the isometry defined by the vertical edge space maps in Table \ref{tabrho}. Let $w \in (\rho \alpha \kappa, \alpha \rho \alpha \kappa)$ be the isomorphism defined by the vertical gauge transformation given in the appendix.

Next we define some diagrams as in \cite{MR2812458}.  We set the coevaluation $$\hpic{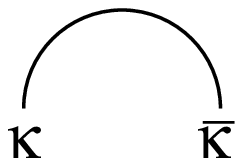} {0.4in} = \sqrt{\beta }r_{\kappa}$$
and let
$$\hpic{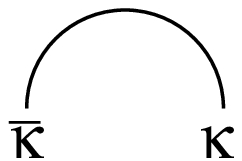} {0.4in}   \sqrt{\beta }\bar{r}_{\kappa}$$
 be the adjoint of the corresponding evaluation. By the duality relation, we have that for any upper vertex $x$ connected to a lower vertex $Y$, $$r_{\kappa}(xx,xYx)\bar{r}_{\kappa}(YY,YxY)=\frac{1}{\beta}   .$$

Also set $$\hpic{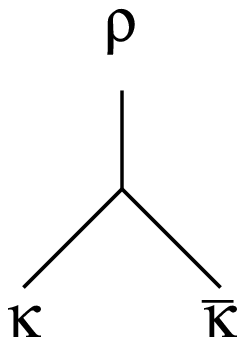} {0.8in} = \sqrt{\frac{\beta}{\beta_1}} v, \quad \hpic{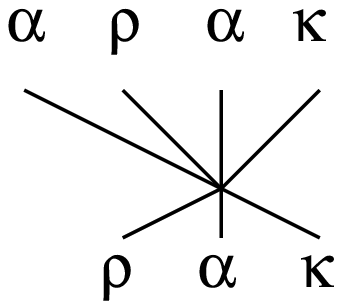} {0.8in} = w .$$

For each of these diagrams define the diagram obtained by reflecting across the horizontal axis to be the adjoint. Define $$\hpic{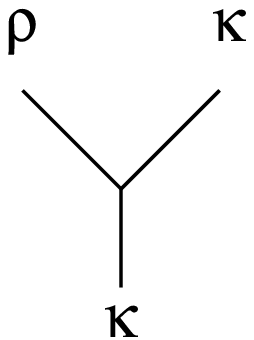} {0.8in} = \hpic{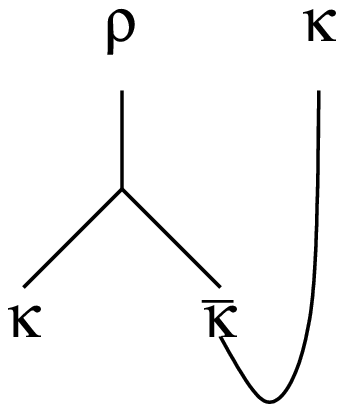} {0.8in} , \hpic{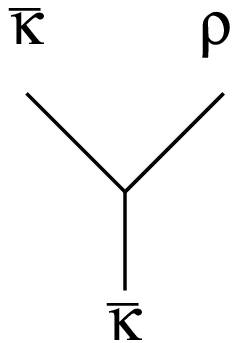} {0.8in} = \hpic{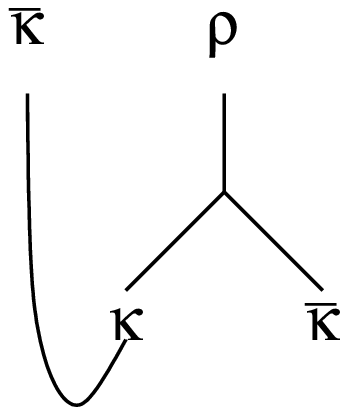} {0.8in} $$
and again define the diagrams obtained by horizontal reflections to be the adjoints. 

Then it is straightforward to check that
$$ \hpic{rhok_k} {0.8in} = \hpic{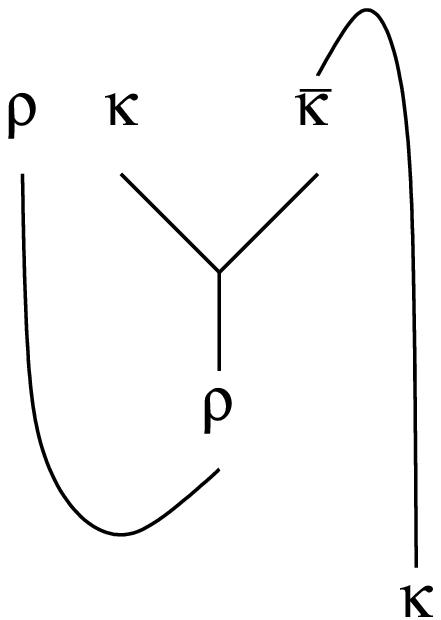} {1.1in} , \quad \hpic{kbar_rho} {0.9in} = \hpic{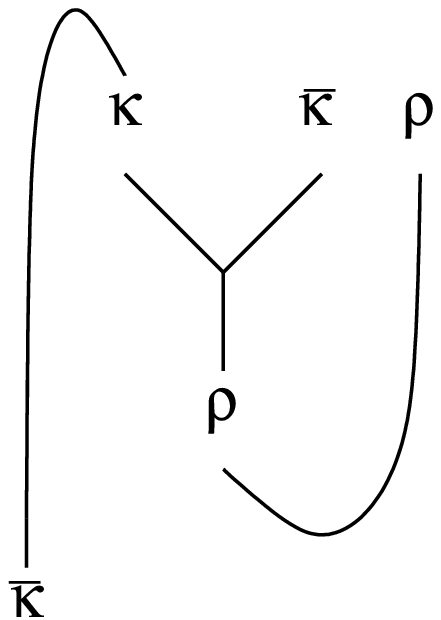} {1.0in} .$$

Next let $$\hpic{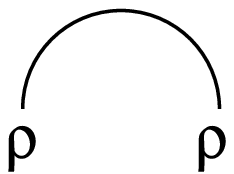} {0.4in} = \frac{\beta_1}{\beta} \hpic{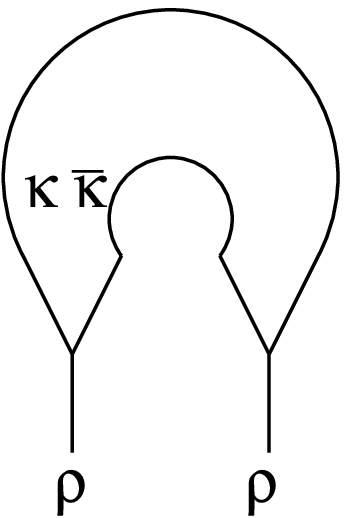} {1.2in} , \hpic{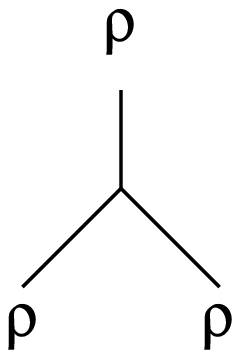} {0.8in} = \frac{\beta_1}{\beta_2} \hpic{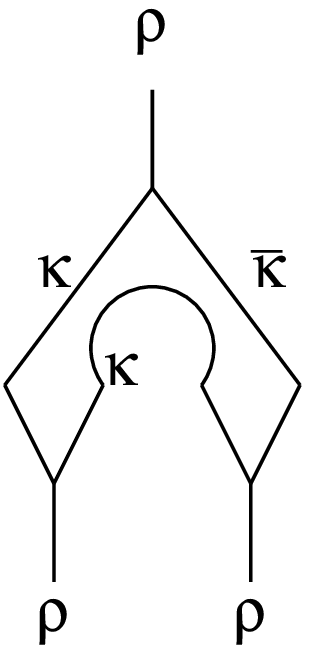} {1.4in} ,$$

$$\hpic{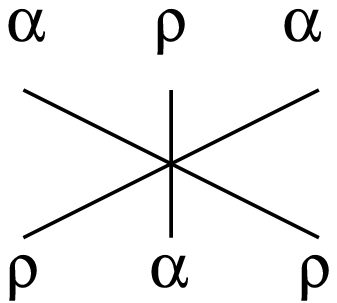} {0.8in} = \frac{\beta_1}{\beta} \hpic{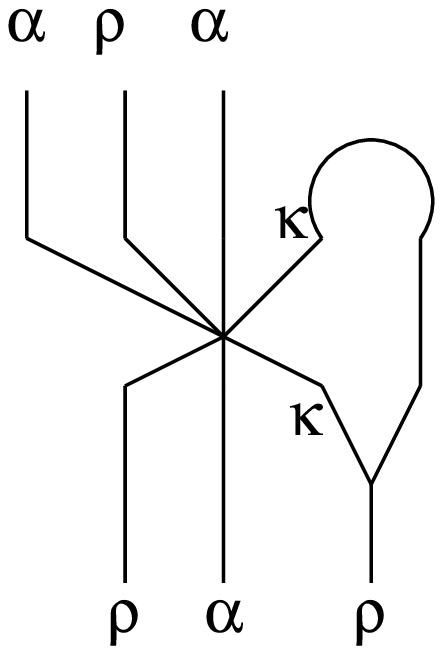} {1.2in} .$$

Again, let each of the diagrams reflected in the horizontal be the adjoint, and again we have $$\hpic{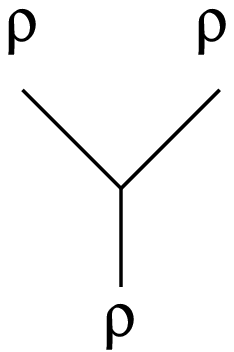} {0.8in}  = \hpic{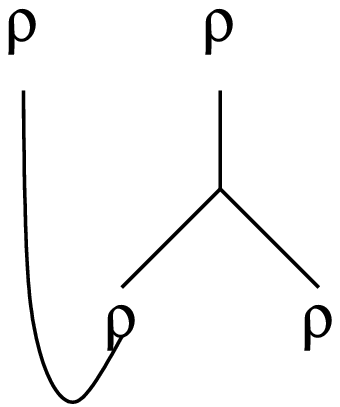} {0.8in} = \hpic{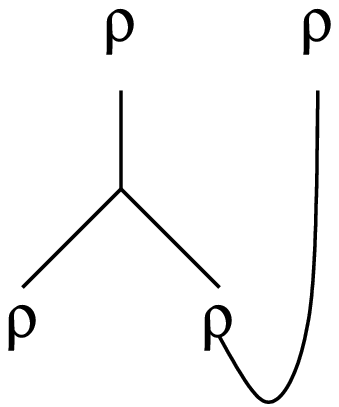} {0.8in} .$$

We now compute a bunch of coefficients for later use.

\begin{lemma} \label{coflem}
We have the following coefficients.
\begin{enumerate}
\item  $$\hpic{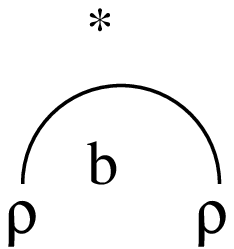} {0.5in} = \hpic{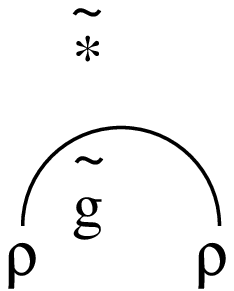} {0.6in} = \beta_1 , \quad \hpic{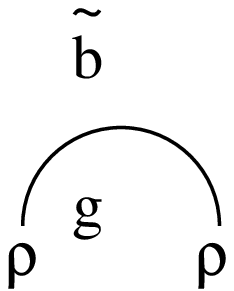} {0.6in} = \hpic{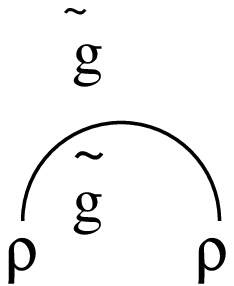} {0.5in} = 1   $$

\item  $$- \hpic{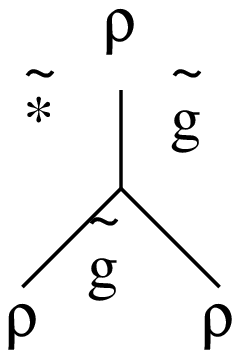} {0.7in } =\hpic{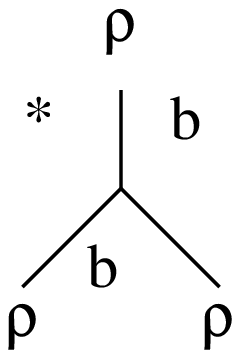} {0.7in} = \beta_2\sqrt{\frac{\beta_1 }{2}}$$
$$ \hpic{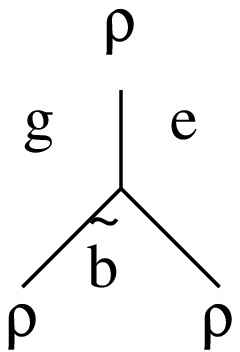} {0.7in} = \hpic{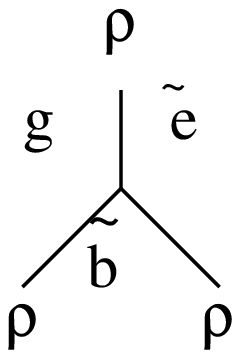} {0.7in} = \hpic{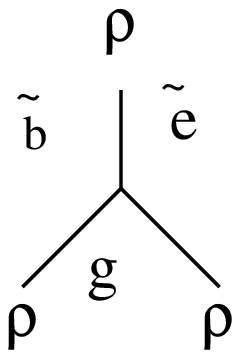} {0.7in} = \hpic{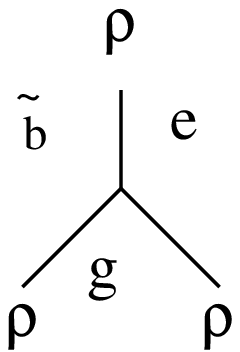} {0.7in} = \sqrt{\frac{\beta_1 }{2}}$$

\item $$\hpic{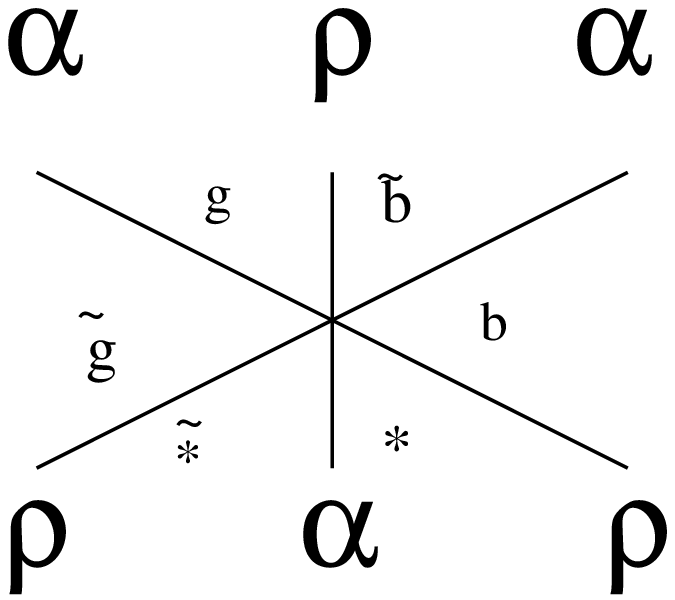} {0.6in} =  \hpic{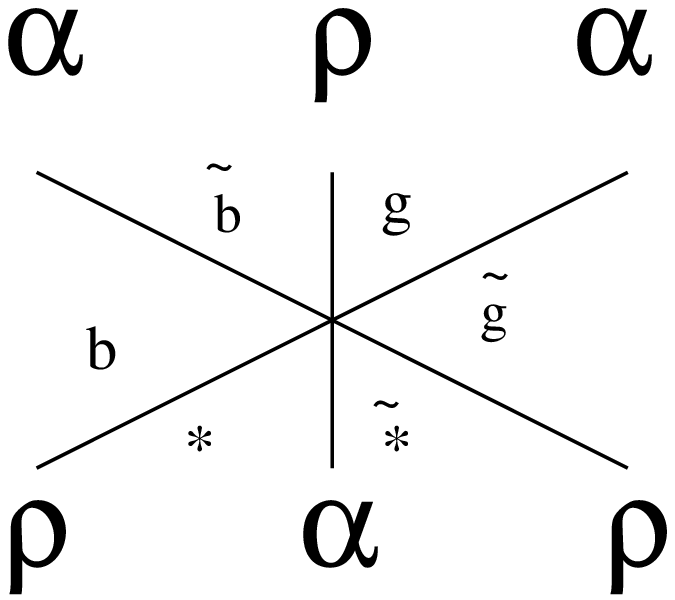} {0.6in} = -\frac{1}{\sqrt{\beta_1}} $$ $$\hpic{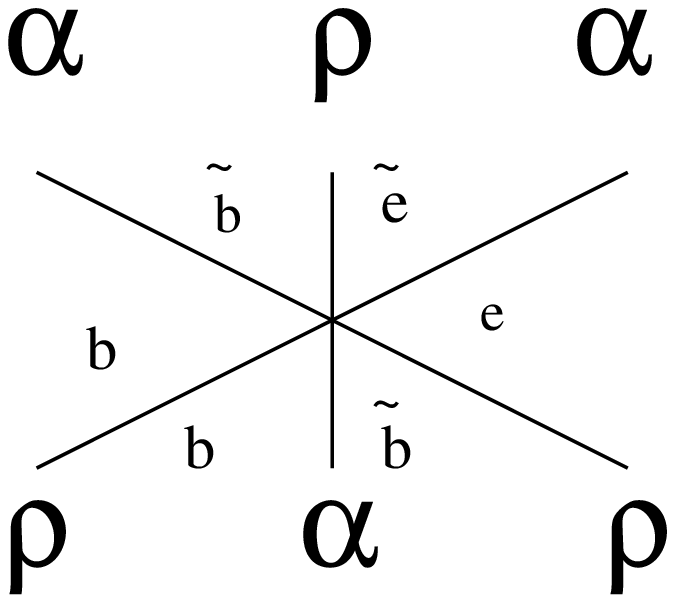} {0.6in } = \hpic{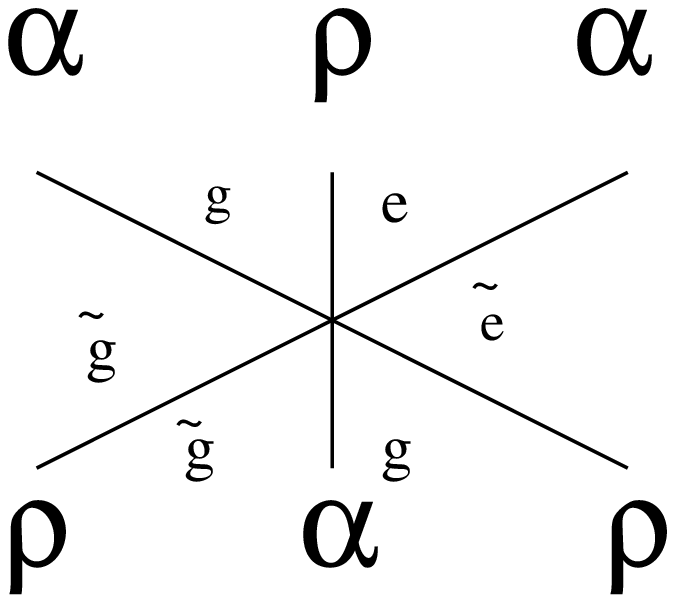} {0.6in} =- \hpic{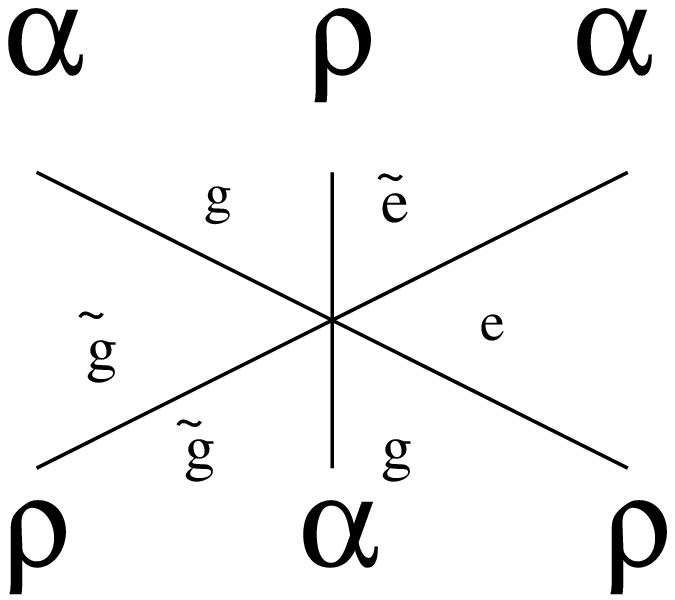} {0.6in}  =  -\hpic{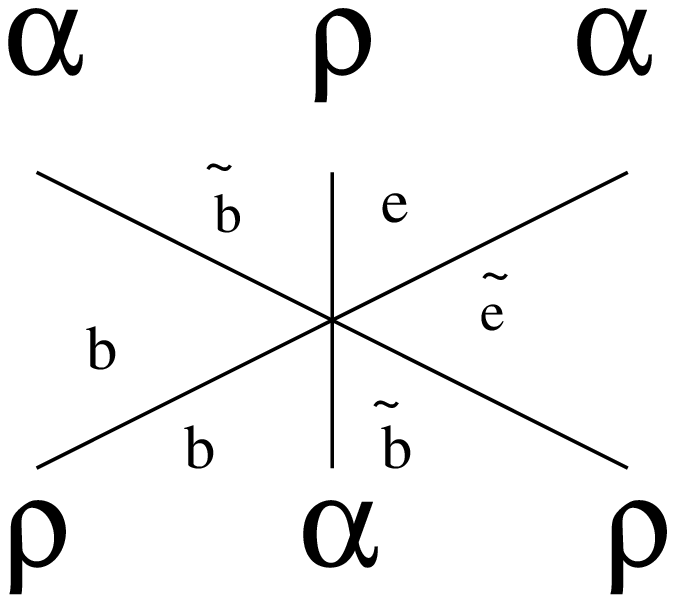} {0.6in}   =\frac{\sqrt{\beta_1}}{\beta_2}$$ $$  \hpic{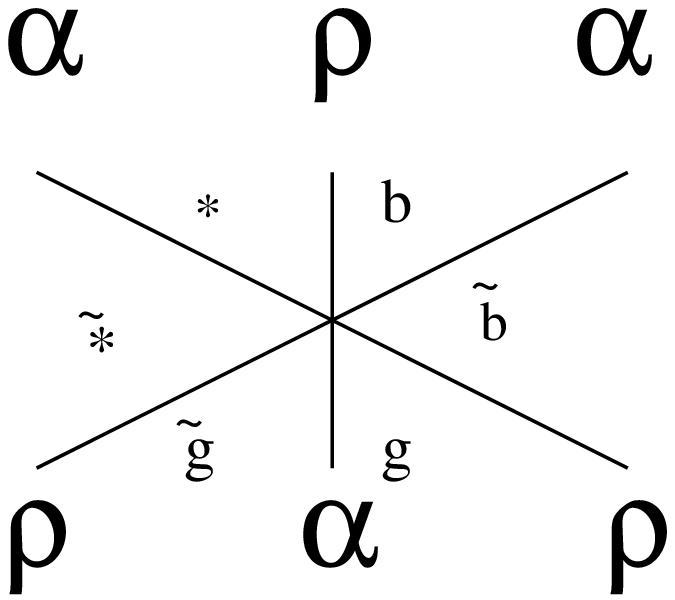} {0.6in} =- \hpic{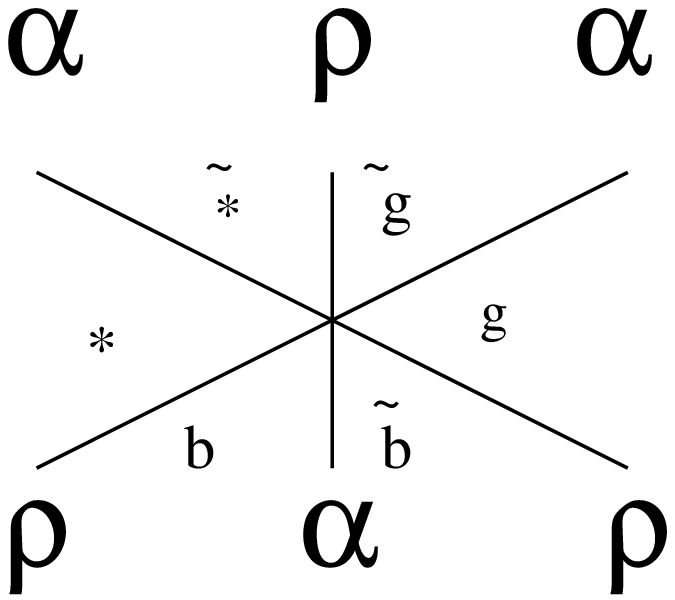} {0.6in} = \sqrt{\beta_1} $$

\end{enumerate}
\end{lemma}

\begin{proof}
These calculations are similar to those in \cite{MR2812458}. Each coefficient diagram is expressed as a product of simpler diagrams, which are evaluated using Tables \ref{tabkap} and \ref{tabrho} (for $r_{\kappa} $ and $v$) and Appendix \ref{bigint} (for $w$). For the convenience of the reader we review one calculation of each type here.
\begin{enumerate}
\item $$\hpic{rrho1} {0.5in} = \frac{\beta_1}{\beta} \hpic{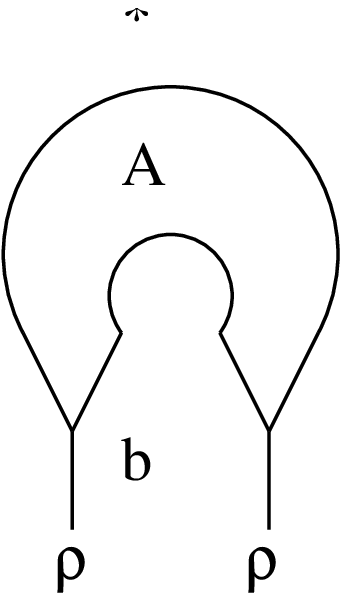} {0.8in} = \frac{\beta_1}{\beta} (\hpic{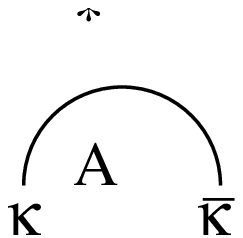} {0.4in} )( \hpic{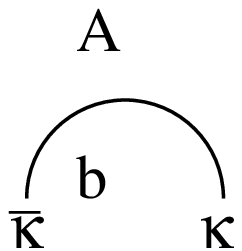} {0.4in} )( \hpic{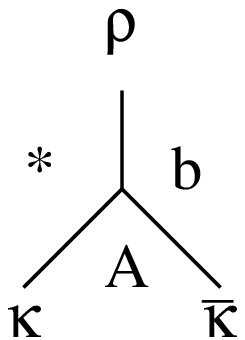} {0.5in} )( \hpic{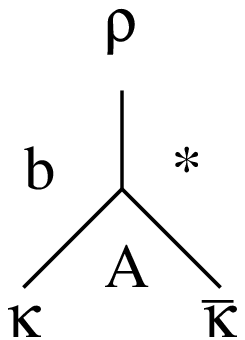} { 0.5in} )$$ $$ = \frac{\beta_1}{\beta}\frac{\beta^2}{\beta_1} r_{\kappa}(**,*A*) \bar{r}_{\kappa}(AA,AbA)v(*b,*Ab)v(b*,bA*)$$ $$= \beta(1)(\frac{\beta_1}{\beta})(1)(1)=\beta_1.$$

%
%
%
%
%
%
%
%

\item $$\hpic{rhorhorho1} {0.6in } = \frac{\beta_1}{\beta_2} \hpic{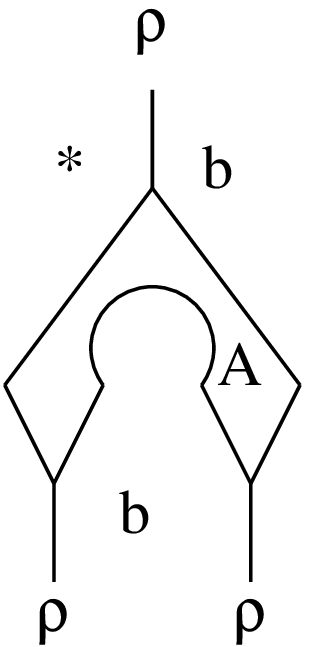} {0.8in} = \frac{\beta_1 }{\beta_2} (\hpic{rbarkappa1} {0.4in} )( \hpic{rhokk1} {0.5in} )( \hpic{rhokk1} {0.5in} )( \hpic{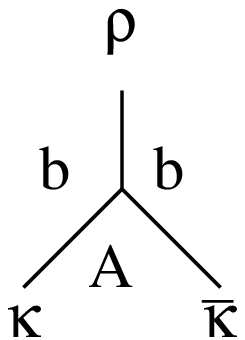} { 0.5in} ) $$ $$= \frac{\beta_1}{\beta_2} \frac{\beta^2}{\beta_1^{\frac{3}{2}}} \bar{r}_{\kappa}(AA,AbA )v(*b,*Ab)v(*b,*Ab)v(bb,bAb)$$ $$=\beta\sqrt{\frac{\beta_1 }{2}}(\frac{\beta_1} {\beta})(1)(1)(-\frac{\beta_2 }{\beta_1 } )=-\beta_2\sqrt{\frac{\beta_1 }{2}} .$$

%
%
%
%
%
%
%
%

%
%
%
%
%
%
%

\item $$\hpic{ararar8} {0.6in} = \frac{\beta_1}{\beta} \hpic{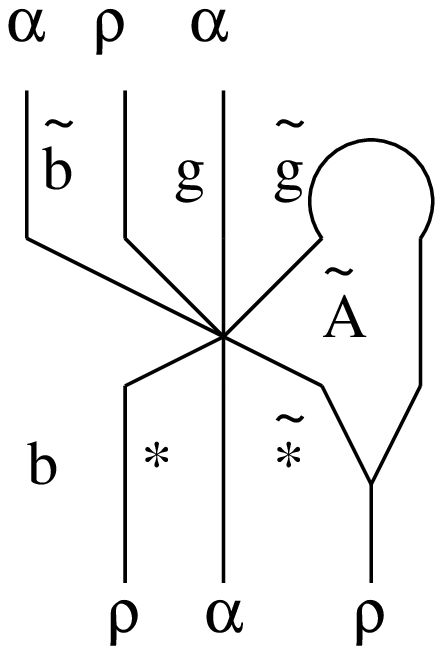} {0.8in} = \frac{\beta_1}{\beta} (\hpic{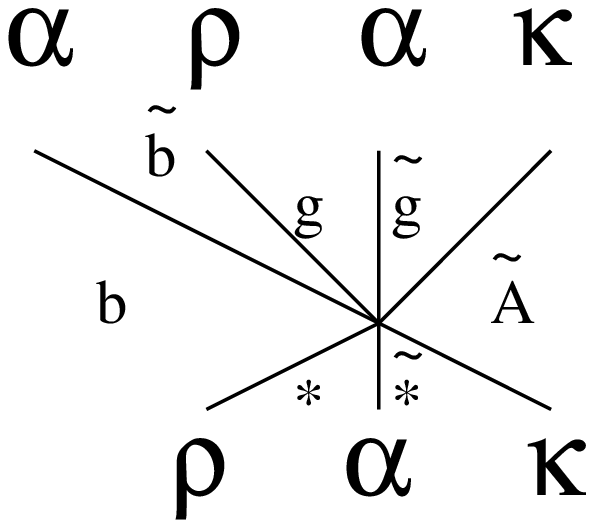} {0.6in} )( \hpic{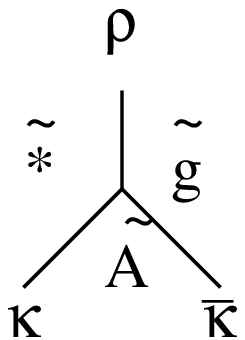} {0.5in} )( \hpic{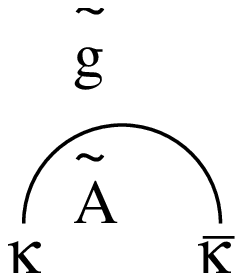} {0.5in} )$$ $$=\frac{\beta_1}{\beta} \frac{\beta}{\sqrt{\beta_1}} w(b\tilde{b}g\tilde{g}\tilde{A},b*\tilde{*}\tilde{A})v(\tilde{*}\tilde{g},\tilde{*}\tilde{A}\tilde{g})r(\tilde{g}\tilde{g},\tilde{g}\tilde{A}\tilde{g} )$$ $$ = \sqrt{\beta_1} (-1) (1 ) (\frac{1}{\beta_1} )=\frac{1}{\sqrt{\beta_1}}.$$

%
%
%
%
%
%
%
%
%
%
%

%
%
%

\end{enumerate}

\end{proof}

%
%
%

%

With these coefficients we can verify the Asaeda-Haagerup relations \ref{algebrarelations}.

\begin{theorem} \label{mnthm}
The Asaeda-Haagerup algebra relations \ref{algebrarelations} are satisfied for $AH+1$.
\end{theorem}

\begin{proof}

\begin{enumerate}
\item The left hand side of each equation is a scalar, so we can evaluate the unique state comptabile with any given edge.
For the first equation we have:\\
 $$\hpic{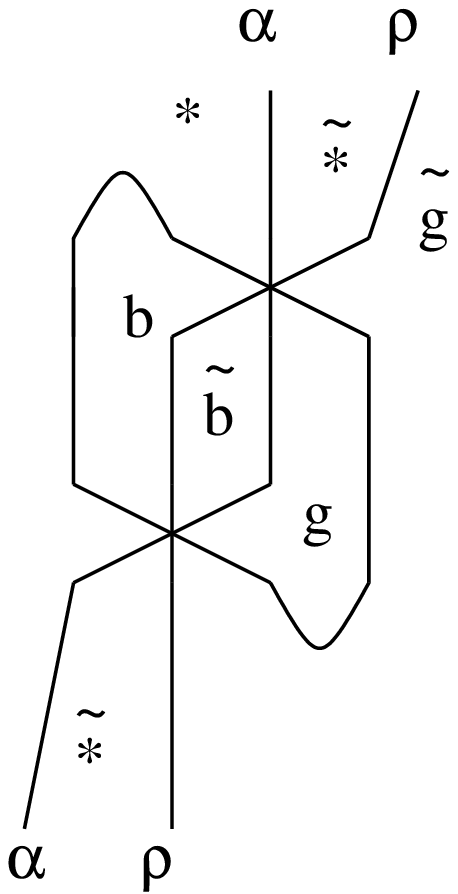} {1.4in} $$ $$=   (\hpic{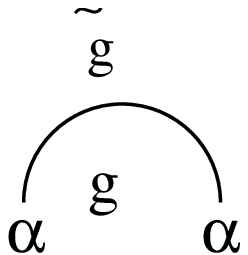} {0.5in} )( \hpic{rrho1} {0.5in } ) ( \hpic{ararar8} {0.5in} ) ( \hpic{ararar14} {0.5in} )$$ $$ = \beta_1 \frac{1}{\sqrt{\beta_1}} \sqrt{\beta_1}=\beta_1 .$$
The second equation is computed similarly, using the unique state for the edge pair
$(*b\tilde{b},*b\tilde{b}) $.

\item Since $dim(\rho,\alpha\rho\alpha\rho\alpha )=1$, we can compare the two sides of the equations using any nonzero coefficient. We choose the coefficient corresponding to the edges $(*b, *\tilde{*}\tilde{g}g\tilde{b}b) $, which admits a unique compatible state for each of the diagrams in the equation. Then we have

$$\hpic{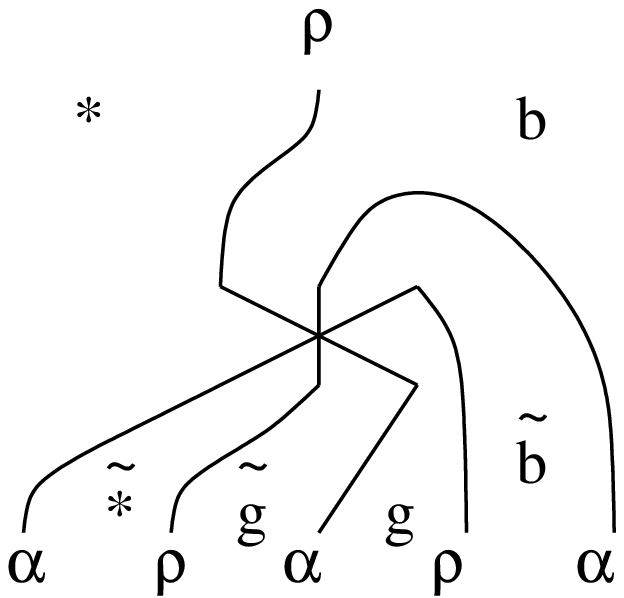} {1.0in} = \hpic{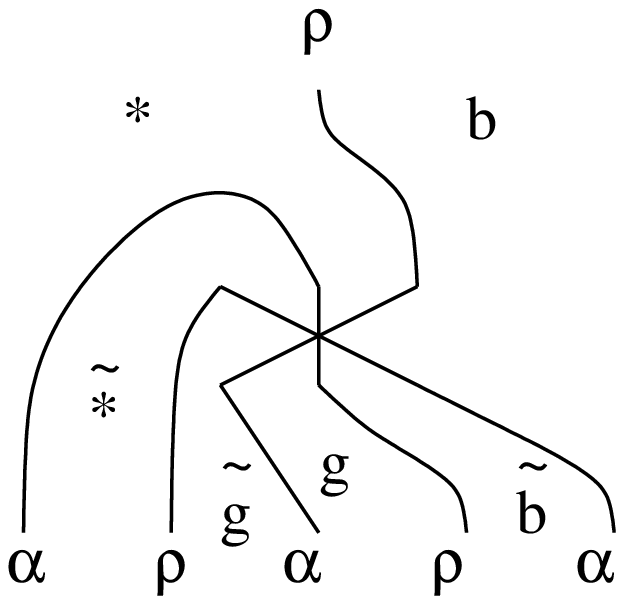} {1.0in} =-\sqrt{\beta_1},$$
where as before we evaluate the states by breaking up each diagram as a product of smaller diagrams.

\item In this case, $dim(\rho \alpha \rho, \alpha\rho\alpha\rho\alpha)=2 $, so it is not sufficient
to compare a single nonzero coefficient. However, using the labeling in Figure \ref{ahp1graphs}, we have that $\rho\alpha \rho=\alpha\rho\alpha+\pi+\alpha\pi $ and $\alpha\rho\alpha \rho\alpha=\rho+\pi+\alpha\pi $, so the common summands are $ \pi$  and $\alpha \pi $.
The vertical graphs for $\rho $ and $\alpha \rho \alpha $ do not have any edges connecting $*$ to $e$ or $\tilde{e} $, so the  
(simple) edges $*b\tilde{b}e $ and $*b\tilde{b}\tilde{e} $ in $\rho \alpha \rho $ must belong to the two summands $\pi $ and $\alpha \pi=\pi \alpha $, and one must belong to each. 
 Therefore to determine an intertwiner it is sufficient to evaluate state diagrams for these two edges.

We have $$-\hpic{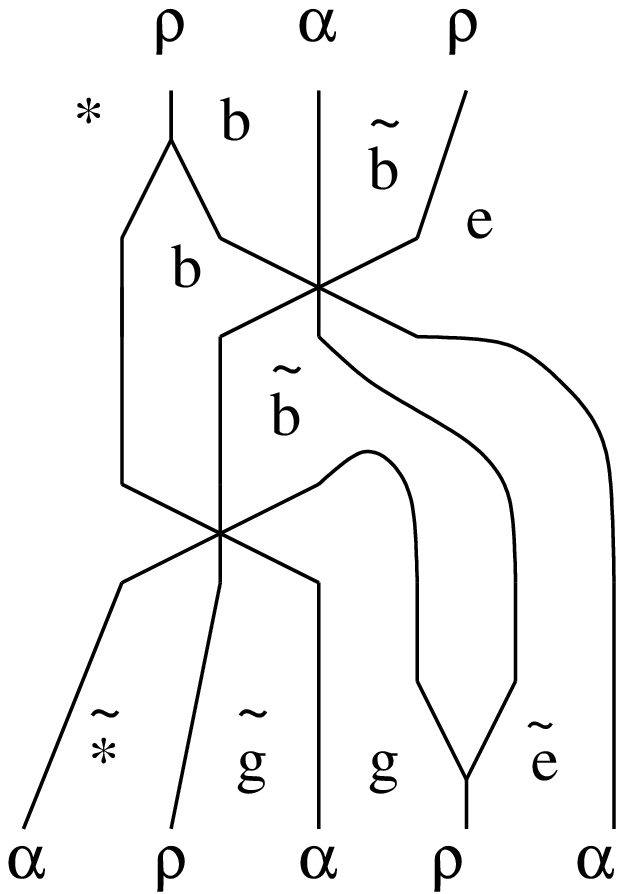} {1in}  =- \hpic{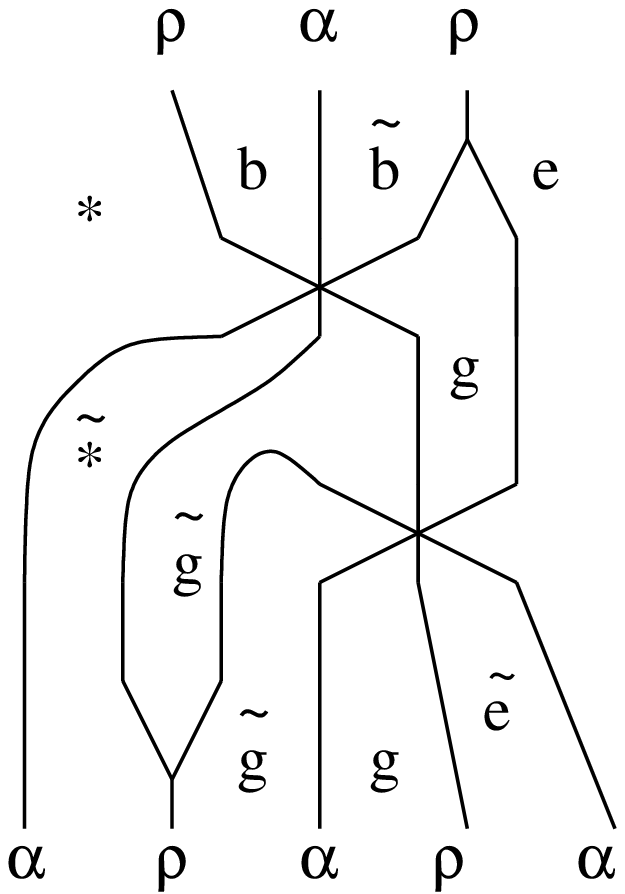} {1in} = \hpic{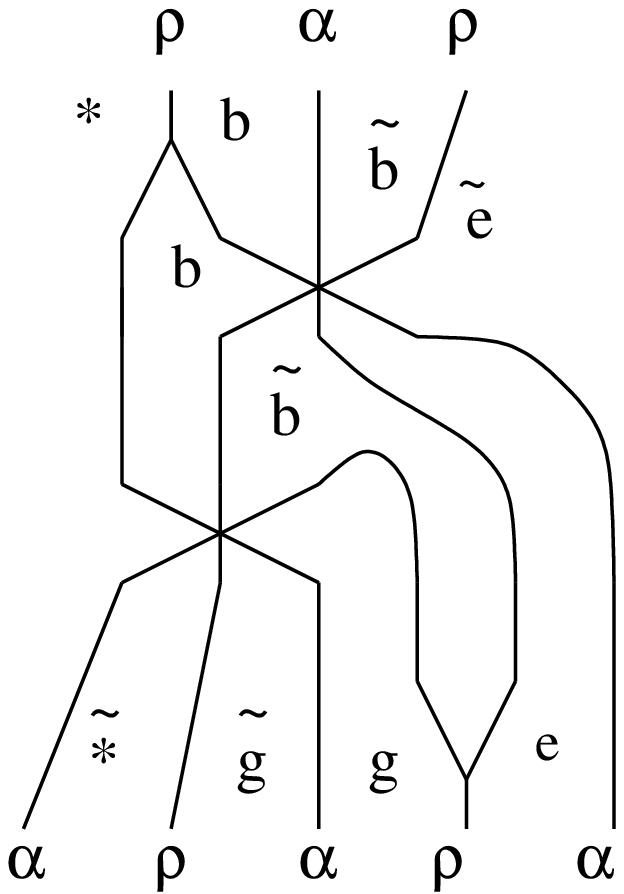} {1in} =  \hpic{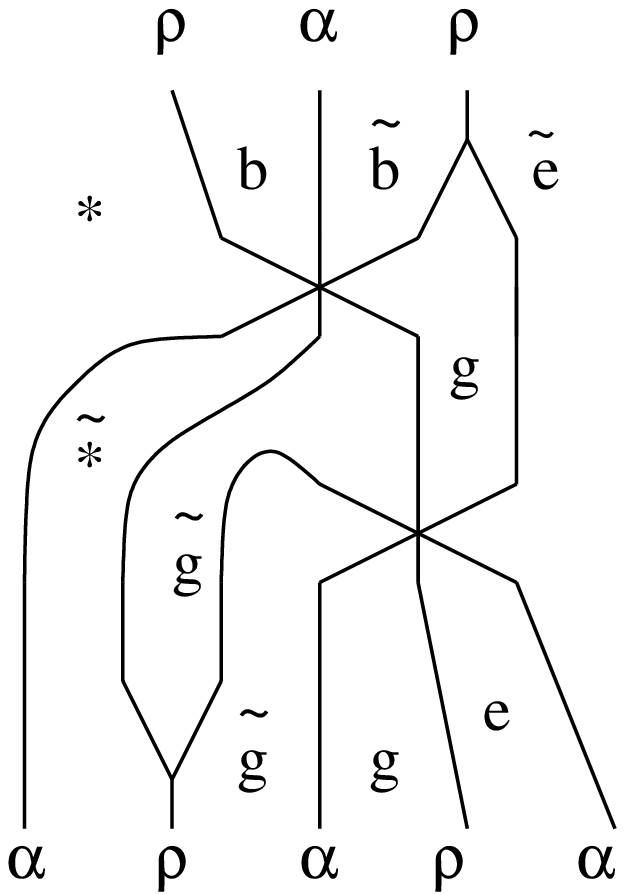} {1in}  =\frac{(\beta_1)^2}{2}.$$

\item The proof is the same as in  \cite{MR2812458} and we omit it.

\end{enumerate}

\end{proof}

\begin{theorem}
There exists a subfactor whose principal and dual graphs are both $$ \hpic{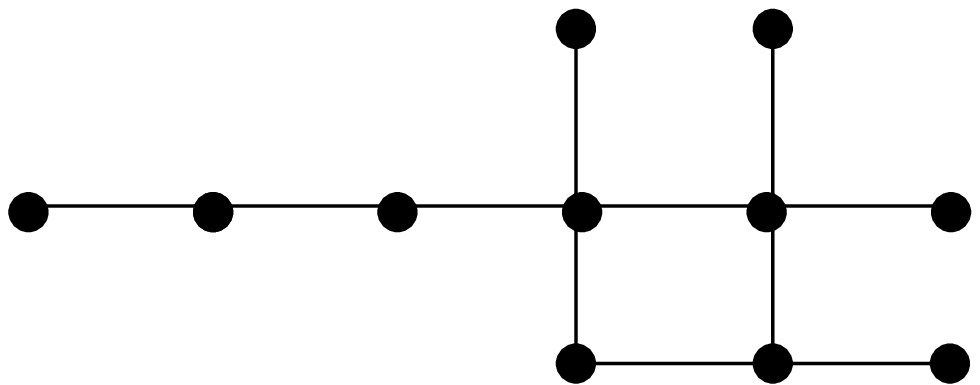} {0.65in}   .$$
\end{theorem}
\begin{proof}
Since the Asaeda-Haagerup algebra relations \ref{algebrarelations} are satisfied for $AH+1$, there is a $Q$-system for $1+\bar{\kappa}\alpha\kappa$ by the same argument as in \cite{MR2812458}, giving a subfactor with index $\frac{9+\sqrt{17}}{2} $. The principal and dual graphs can be easily computed by standard fusion rule calculations.
\end{proof}
We call this subfactor the $AH+2$ subfactor.

\begin{corollary}
There is an irreducible, noncommuting quadrilateral of factors $\begin{array}{ccc}
P&\subset &M \cr
\cup& &\cup \cr
N&\subset &Q
\end{array}$ such that $P \subset M $ and $Q \subset M $ are both the $AH+1$ subfactor and $N \subset P $ and $N \subset Q$ are both the $AH+2$ subfactor.
\end{corollary}

\section{A different proof for the existence of $AH+1$ and $AH+2$}
A new construction of the Asaeda-Haagerup subfactor was given in \cite{AHCat2}. First a new subfactor with index $5+\sqrt{17} $ and principal graph $$\hpic{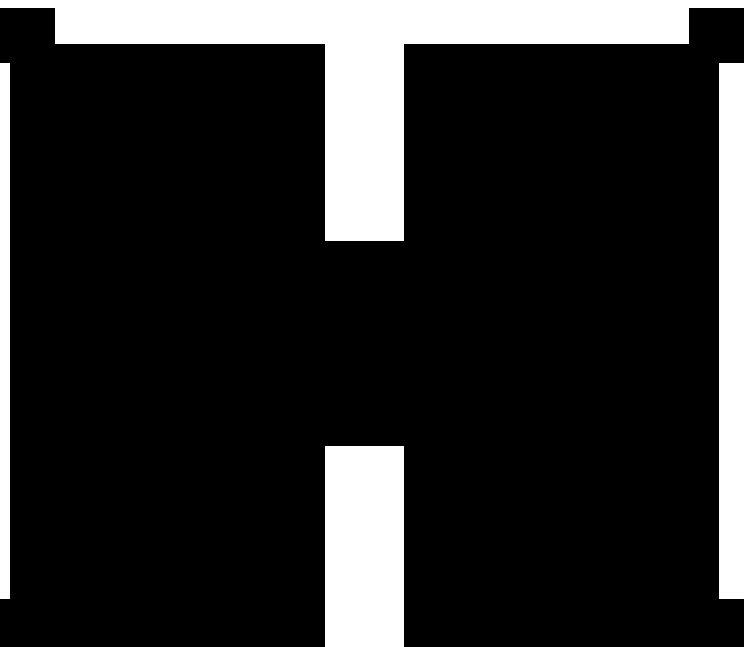} {1.2in}  ,$$ which we call $2AH$,  was constructed from endomorphisms of a Cuntz algebra. Then it was shown that the dual even part of this subfactor has the same fusion rules as the principal even part of the Asaeda-Haagerup subfactor. Finally, because this dual even part contains a self-dual simple object $\rho $ satisfying $\rho^2=1+\rho+\pi $, with $\pi $ irreducible, $1+\rho $ must admit a Q-system by the recognition theorem \cite[Theorem 3.4]{GSbp}. The Q-system for $1+\rho$ gives a subfactor with the Asaeda-Haagerup principal graph (by uniqueness of the connection for this graph, it is the same subfactor constructed by Asaeda and Haagerup).

The condition $\rho^2=1+\rho+\pi $ corresponds to $4$-supertransitivity of the Asaeda-Haagerup subfactor - this means that the principal graph has a single branch of at least $4$ edges emanating from the vertex labeled by $1$ before any branching out occurs. (The Asaeda-Haagerup subfactor is in fact $5$-supertransitive). The $AH+1$ and $AH+2$ subfactors are only $3$-supertransitive, so the recognition theorem of \cite{GSbp} does not apply. Nevertheless, it is possible to deduce the existence of $AH+1$ and $AH+2$ from the existence of $2AH$ and $AH$ using combinatorics of the Brauer-Picard groupoid.

In \cite{AHCat2}, the Brauer-Picard groupoid of the Asaeda-Haagerup fusion categories was described. There are six different fusion categories in the Morita equivalence class, denoted in caligraphic font as $\mathcal{AH}_1-\mathcal{AH}_6 $.

 The relationship of these categories to the small-index subfactors is as follows: $\mathcal{AH}_1$ is the common dual even part of the $AH$, $AH+1$, and $AH+2 $ subfactors, $\mathcal{AH}_2$ is the principal even part of the $AH$ subfactor, $\mathcal{AH}_3$ is the principal even part of $AH+1 $, and $\mathcal{AH}_4$ is the principal even part of the new $2AH$ subfactor. The principal even part of $AH+2$ is also $\mathcal{AH}_1 $. This information is summarized in Figure \ref{smallind}.

\begin{figure}
\begin{centering}
\begin{tikzpicture}

\draw[<->,thick] (0,0.4) --(0,3.6); 
\draw[<->,thick] (0.6,4) --(3.6,4);
\draw[<->,thick] (0.6,0) --(3.6,0);
\node at (0,4) {$\mathcal{AH}_1$};
\node at (0,0) {$\mathcal{AH}_2$};
\node at (4.2,4) {$\mathcal{AH}_3$};
\node at (4.2,0) {$\mathcal{AH}_4$};
\node at (2.2,4.3) {$AH+1$};
\node at (2.2,-0.3) {$2AH$};
\node at (-0.5,2.3) {$AH$};
\draw[<->,thick] (0,4.4) arc (0:330:1.3); 
\node at (-3.5,4.5) {$AH+2$};
\end{tikzpicture}
\caption{Some small index subfactors in the Brauer-Picard groupoid}
\label{smallind}
\end{centering}
\end{figure}
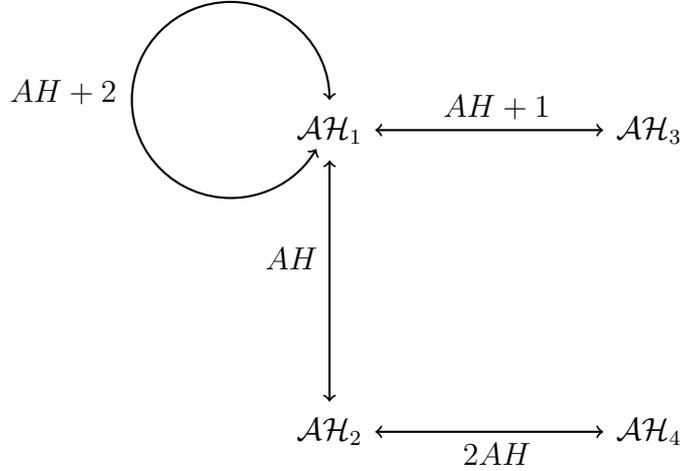
The Brauer-Picard group is $\mathbb{Z}_2 \times \mathbb{Z}_2 $, and all four invertible $\mathcal{AH}_i-\mathcal{AH}_j $ bimodule categories for $1 \leq i,j \leq 3 $ were described in \cite{GSbp}. However the calculations behind these results used the existence of $AH+1$ and $AH+2$, and preceded the discovery of $2AH$.


Now we will only assume the existence of $2AH$, and as a consequence $AH$. Then we have three fusion categories which arise as even parts of these two subfactors, namely, $\mathcal{AH}_1$, $\mathcal{AH}_2$, and $ \mathcal{AH}_4 $. 

The fusion category $ \mathcal{AH}_4 $ contains eight simple objects. There is a tensor subcategory equivalent to $\text{Vec}_{\mathbb{Z}_4 } $, with simple objects $\alpha_i, \ i \in \mathbb{Z}_4$ and a simple object $\xi $ satisfying 
$$\alpha_i \xi=\xi \alpha_{-i}, \quad \xi^2=1+2 \sum_{i \in \mathbb{Z}_4} \limits \alpha_{i} \xi.$$
We have $$dim(\xi)=d:=4+\sqrt{17} .$$
There is a Q-system for $1+\xi$, which gives the $2AH$ subfactor. The corresponding dual Q-system in $\mathcal{AH}_2 $ is $1+\alpha \pi $ (where we use the labeling from Figure \ref{ahgraphs}.) 

\begin{lemma}
The Brauer-Picard group of the Asaeda-Haagerup fusion categories contains $\mathbb{Z}_2 \times \mathbb{Z}_2 $.
\end{lemma}
\begin{proof}
In fact as noted above, the Brauer-Picard group of the Asaeda-Haagerup fusion categories is isomorphic to $\mathbb{Z}_2 \times \mathbb{Z}_2 $, which is \cite[Theorem 6.7(b)]{GSbp}. However, the proof there used the existence of $AH+1$ and $AH+2$, so we need a different approach here. 

The construction of $2AH$ in \cite{AHCat2}  proceeded by explicitly constructing the endomorphisms $\alpha_i $ and $\xi $ in $\mathcal{AH}_4 $ on the von Neumann algebra closure of a Cuntz algebra and then verifying that $1+\xi $ admits a Q-system. 

However, there are actually two inequivalent Q-systems for $1+\alpha_i \xi $ for each $i$. By enlarging the Cuntz algebra, one can explicitly construct a graded extension of $ \mathcal{AH}_4$ by $\mathbb{Z}_2 \times \mathbb{Z}_2 $ generated by an outer automorphism which switches the two Q-systems for $1+\xi$ and an outer automorphism which switches $ \xi$ and $\alpha_1 \xi $. (This result was announced in \cite{AHCat2} although the details of the construction do not appear there.) 
\end{proof}

In the following lemmas, we will need to perform some combinatorial calculations in the Brauer-Picard groupoid, following the methods of \cite{GSbp}. We briefly explain some of the notation from there, which we also employ here. The Grothendieck ring of each fusion category $\mathcal{AH}_4 $ is denoted by $AH_i$ (not to be confused with $AH$, $AH+1$, and $AH+2$, which refer to subfactors). 

In an arXiv supplement to \cite{GSbp} (see arxiv:1202.4396),  there are text files \textit{AH1Modules}, \textit{AH2Modules},  and \textit{AH3Modules}, which give lists of (right) fusion modules over the fusion rings $AH_1 $, $AH_2$, and $AH_3$, respectively. Each fusion module is given as a list of non-negative integer matrices. The $ij^{th} $ entry of the $k^{th}$ matrix gives $(\kappa_k \xi_i,\kappa_j) $, where the $\xi_i $ are the basis elements of the fusion ring and the $\kappa_j $ are the basis elements of the fusion module.

We use $n_j $ to refer to the $n^{th} $ fusion module on the list of fusion modules for $AH_i$. For example, $16_2 $ refers to the $16^{th}$ fusion module on the list of $AH_2$ fusion modules given in the file \textit{AH2Modules}. The text file \textit{Bimodules}, also in the arXiv supplement, gives lists of $AH_i-AH_j$ fusion bimodules for each $1 \leq i,j \leq 3 $. The bimodule $n_{ij} $ refers to the $n^{th} $ bimodule on the list of $AH_i-AH_j$ bimodules in the file \textit{Bimodules}.

We say that a fusion module $n_i$ is realized if there is a $\mathcal{AH}_i$ module category whose fusion module is $n_i $, and $n_i$ is realized uniquely if there is a unique such module category; and similarly for fusion bimodules and bimodule categories. 

If a fusion module $n_i $ is realized by a module category $\mathcal{M}_{ \mathcal{AH}_i}$, then one can read from the data of $n_j$ the list of objects which have algebra/Q-system structures whose categories of modules are equivalent to $\mathcal{M} $ - such objects are described by the $j^{th} $ columns of the $ j^{th}$ matrices in the list of matrices for $ n_i$. We will say that a Q-system $\gamma \in \mathcal{AH}_i $ is associated to a fusion module $n_i$, or vice versa, if the module category of $\gamma$ realizes $n_i $. Similarly, $\gamma $ is associated to $n_{ji} $ if the bimodule category of $\gamma $ realizes $n_{ji} $. 

For the small index subfactors, it is easy to see which fusion modules they correspond to. For example, the subfactor $AH$ corresponds to a Q-system for $1+\rho$ in $\mathcal{AH}_2 $, so we look for a fusion module over $AH_2$ which has a matrix containing a column with $1$'s as the entries corresponding to the basis elements for $1$ and $ \rho$ and with $0$'s for the other entries; the only such fusion module is $16_2$.

The key idea in the following calculations is the notion of multiplicative compatibility of bimodule categories. A triple $(l_{ij},m_{jk},n_{ik} ) $ is said to be multiplicatively compatible if $n_{ik}$ passes certain combinatorial obstruction tests for being realized by the tensor product of bimodule categories realizing $l_{ij} $ and $m_{jk} $ (see \cite{GSbp} for details). The file \textit{BimoduleCompatibility} gives, for each pair $(l_{ij},m_{jk}) $ with $1 \leq i,j,k \leq 3 $, the set of $AH_i-AH_k $ fusion bimodules which form compatible triples with the pair. 

The notation used is $$l_{ij} \cdot m_{jk}=\{x_{ik},y_{ik},...\} .$$ If the right hand side is a singleton set, we say that $l_{ij} \cdot m_{jk}$ is a unique multiplication and suppress the braces. There are certain obvious facts that can be deduced from the multiplicative compatibility tables. For example, if  $l_{ij} \cdot m_{jk}=\{ \}$ then $l_{ij}$ and $m_{jk} $ cannot both be realized. If 
$l_{ij} \cdot m_{jk}=n_{ik}$ is a unique multiplication and $l_{ij}$ and $m_{jk} $ are both realized then so is $n_{ik}$. 

We wll also need the fusion bimodule lists and multiplicative compatibility rules for $AH_1-AH_4 $ and $AH_4-AH_1$ bimodules; we include these in the arXiv submission in the text file \textit{AH1-AH4\_Bimodules}.

\begin{lemma}
There are two invertible $\mathcal{AH}_1-\mathcal{AH}_4 $ bimodule category realizing the fusion bimodule $8_{14} $, and two realizing the fusion bimodule $9_{14} $. 
\end{lemma}
\begin{proof}
There is an invertible $\mathcal{AH}_1-\mathcal{AH}_2 $ bimodule category associated to a Q-system for $1+\rho$ in $\mathcal{AH}_2 $ (coming from the subfactor $AH$), and an invertible $\mathcal{AH}_2-\mathcal{AH}_4 $ bimodule category associated to a Q-system for $1+\alpha \pi$ in $\mathcal{AH}_2 $ (coming from the subfactor $2AH$). Since $(1+\rho,1+\alpha \pi) =1$, this means that there is an invertible $\mathcal{AH}_1-\mathcal{AH}_4 $ bimodule category associated to a Q-system of dimension $dim(1+\rho)dim(1+\alpha\pi)=\frac{1+d}{2}(1+d) =1+5d$. 
By inspecting the list of $\mathcal{AH}_1-\mathcal{AH}_4 $ bimodules, we see that the only two candidates are $8_{14} $ and $9_{14} $.
Also, since two of the four outer automorphisms of $\mathcal{AH}_4 $ fix $\rho $ and two send $\rho $ to $\alpha_1 \rho $, any bimodule category realizing $8_{14} $ is sent to a bimodule category realizing $8_{14}$ by two of the four outer automorphisms, and is sent to a bimodule category realizing $9_{14} $ by the other two outer automorphisms; and similarly for bimodule categories realizing $9_{14} $. So there must be at least two invertible $\mathcal{AH}_1-\mathcal{AH}_4 $ realizing each of $8_{14} $ and $9_{14}$. 
\end{proof}

By considering the opposite bimodule categories, we see that there are also two $\mathcal{AH}_4-\mathcal{AH}_1 $ realizing each of $8_{41} $ and $9_{41}$. 
\begin{lemma}
There are invertible bimodule categories realizing $10_{11} $, $12_{11} $, $13_{11}$, and $14_{11}$.
\end{lemma}
\begin{proof}
Looking at the multiplicative compatibility lists for $AH_1-AH_4$ fusion bimodules with 
$AH_4-AH_1$ fusion bimodules in the file \textit{AH1-AH4\_Bimodules}, we find that

$$8_{14}  \cdot  8_{41} = \{ 12_{11} , 14_{11} \} , \quad
8_{14} \cdot 9_{41} = \{ 8_{11} , 10_{11} , 13_{11} \}  .$$

This means that that there is a subgroup of the Brauer-Picard group of $\mathcal{AH}_1 $ which is isomorphic to $\mathbb{Z}_2 \times \mathbb{Z}_2$ and contains two bimodule categories realizing fusion bimodules from the set
 $  \{ 12_{11} , 14_{11} \} $ and two bimodule categories realizing bimodules from the set 
$  \{ 8_{11} , 10_{11} , 13_{11} \}$.

The fusion bimodule $14_{11}$ is realized uniquely by the trivial bimodule category (which is the identity in the Brauer-Picard groupoid), since $\mathcal{AH}_1 $ has no outer automorphisms by the argument in \cite{GSbp}.  
Looking at the multiplicative compatibility lists for $\mathcal{AH}_1-\mathcal{AH}_1 $ bimodule categories (in the file \textit{BimoduleCompatibility} from \cite{GSbp}), we see that $a_{11}\cdot a_{11} $ for $a=8,10,12,13$ is compatible only with $14_{11}$, so each of the fusion bimodules $12_{11}$, $10_{11}$ ,$12_{11}$, $13_{11}$, if realized, is realized uniquely by \cite[Lemma 6.4]{GSbp}. Also, $12_{11}$ is realized since it is the only other member of the multiplicative compatibility list for $8_{14}\cdot 8_{41} $ alongside $14_{11}$.

Finally, we show that $8_{11}$ cannot be realized. The multiplicative compatiblity list for $12_{11}\cdot 8_{11} $ is $\{ 8_{11}, 10_{11} \}  $. So if $8_{11}$ is realized, then since it is necessarily realized uniquely, then the bimodule categories realizing $12_{11}$ and $8_{11}$  have a tensor product realizing $10_{11}$. Then 
since the multiplicative compatiblity list for $12_{11}\cdot 10_{11} $ is $\{13_{11}\}$, the fusion bimodule $13_{11}$ is realized as well. But there is no group structure on any order $4$ subset of the $5$ fusion bimodules which is compatible with these multiplication constraints. So $8_{11}$ is not realized, and the other four are.
\end{proof}

\begin{corollary}
The subfactor $AH+2$ exists.
\end{corollary}
\begin{proof}
The $AH+2$ subfactor corresponds to a Q-system associated to a bimodule category realizing the fusion bimodule $12_{11}$.
\end{proof}

\begin{lemma}
There is a fusion category $\mathcal{AH}_3 $ with Grothendieck ring $AH_3$, and an $\mathcal{AH}_2-\mathcal{AH}_3 $-bimodule category realizing $ 6_{23}$. 
\end{lemma}
\begin{proof}
We consider the list of fusion modules over $AH_2$ from the text file \textit{AH2Modules} in \cite{GSbp}. The subfactor $AH$ is associated to the fusion module $16_2$. We see from looking at the list of Q-systems associated to $16_2$ that there is a Q-system containing both invertible objects $1$ and $\alpha$ in $\mathcal{AH}_2$. Therefore there is a Q-system structure on $1+\alpha$. The possible fusion modules corresponding to such a Q-system are $14_2$ and $15_2$. We compute the possible dual fusion rings for these two fusion modules using the methods of \cite{AHCat2} and find that the only possible dual ring is $AH_3$, which must therefore be the Grothendieck ring of the dual category $\mathcal{AH}_3 $ of $(1+\alpha)-(1+\alpha) $ bimodules in $\mathcal{AH}_2 $. Finally, we check the list of $AH_2-AH_3$ fusion bimodules, and the only one compatible with the Q-system $1+\alpha$ is $6_{23}$.
\end{proof}

\begin{lemma}
There is an $\mathcal{AH}_1-\mathcal{AH}_3 $ bimodule category realizing $6_{13} $.
\end{lemma}
\begin{proof}
The odd part of the (dual) Asaeda-Haagerup subfactor realizes the fusion bimodule $9_{12} $. Multiplicative compatibility rules for $9_{12} $ and $6_{23} $ show that there is a $\mathcal{AH}_1-\mathcal{AH}_3 $ bimodule category realizing $7_{13} $. Then multiplicative compatibility for $13_{11}$ with $7_{13}$ shows that there is a bimodule category realizing $6_{13} $.

\end{proof}

\begin{corollary}
The subfactor $AH+1$ exists.
\end{corollary}

\begin{proof}
The $AH+1$ subfactor corresponds to a Q-system associated to a bimodule category realizing the fusion bimodule $6_{13}$.
\end{proof}

\appendix

\section{The vertical gauge transformation between $\rho \alpha \kappa $ and $\alpha \rho \alpha \kappa $  for $AH+1$ } \label{bigint}

In this appendix we give the details for the vertical gauge transformation between $\rho \alpha \kappa $ and $\alpha \rho \alpha \kappa $  for $AH+1$. We denote composite edges in the $4$-graphs of $\rho \alpha \kappa $ and $\alpha \rho \alpha \kappa $ by words in the appropriate vertices. Thus for example
$\tilde{c}f_1f6$ denotes the (right) vertical edge in $\rho \alpha \kappa $ which is composed
of the edge $\tilde{c}f_1$ in the $4$ graph of $\rho $ (with the subscript distinguishing among the two edges connecting $\tilde{c} $ and $f$), followed by the edge $ff$ in the $4 $-graph of $\alpha $, followed by the edge $f6$ in the $4$-graph of $\kappa $. Then the vertical gauge transformation data consists of a square matrix  for each pair of initial and terminal vertices of edges in the $4$-graphs of $\rho \alpha \kappa $ and $\alpha \rho \alpha \kappa $, with columns indexed by edges in $\rho \alpha \kappa $ and rows indexed by edges in $\alpha \rho \alpha \kappa $. For $AH+1 $, there are $25$ $1\times 1$ matrices, $14$ $2\times 2$ matrices, $10$ $3\times 3$ matrices, $3$ $4\times 4$ matrices, and a $5\times 5$ matrix.

We now list the data, which was found using similar methods to the calculations in \cite{MR1686551}. The fact that this is indeed a vertical gauge transformation is verified in the accompanying Mathematica notebook.

The following $1\times 1$ gauge matrices have entries with the value $1$: 
$$\tilde{a}-6,\tilde{a}-4,\tilde{b}-\tilde{A},\tilde{e}-\tilde{A}, \tilde{e}-A,a-5,c-3,\tilde{c}-2,\tilde{*}-F,b-C  .$$

The following $1\times 1$ gauge matrices have entries with the value $-1$: 
$$\tilde{a}-5,a-4,a-6,\tilde{b}-C,*-F, e-\tilde{A}, \tilde{c}-3,c-2,$$ $$b-\tilde{A}, f-1,e-A ,g-A,g-\tilde{C},\tilde{g}-A,\tilde{g}-\tilde{C} .$$

The larger matrices are as follows:

\renewcommand{\arraystretch}{1.5} 
\resizebox{\linewidth}{!}{%
$
\begin{blockarray}{ccc}
ca\tilde{a}6 & cff6 \\
\begin{block}{(cc)c}
 -\frac{1}{4} \sqrt{\frac{1}{2} \left(23-\sqrt{17}\right)} & \frac{1}{8}
   \left(-1-\sqrt{17}\right) &c\tilde{c}f_1f6\\
 \frac{1}{8} \left(1+\sqrt{17}\right) & -\frac{1}{4} \sqrt{\frac{1}{2}
   \left(23-\sqrt{17}\right)} &c\tilde{c}f_1f6\\
    \end{block}
\end{blockarray},
\quad
\begin{blockarray}{ccc}
\tilde{c}f_1f6 & \tilde{c}f_2f6  \\
\begin{block}{(cc)c}
\frac{1}{4} \sqrt{\frac{1}{2} \left(23-\sqrt{17}\right)} & \frac{1}{8}
   \left(-1-\sqrt{17}\right) & \tilde{c} ca\tilde{a}6\\
 \frac{1}{8} \left(1+\sqrt{17}\right) & \frac{1}{4} \sqrt{\frac{1}{2}
   \left(23-\sqrt{17}\right)}&\tilde{c}cff6 \\
    \end{block}
\end{blockarray}
$
}

\renewcommand{\arraystretch}{1.5} 
\resizebox{\linewidth}{!}{%
$
\begin{blockarray}{ccc}
g\tilde{e}eC & g\tilde{b}bC \\
\begin{block}{(cc)c}
 -\frac{1}{4} \sqrt{7-\sqrt{17}} & \frac{\sqrt{9+\sqrt{17}}}{4} & g\tilde{g}ddC \\
 \frac{\sqrt{9+\sqrt{17}}}{4} & \frac{\sqrt{7-\sqrt{17}}}{4} & g\tilde{g}\tilde{e}eC \\
    \end{block}
\end{blockarray},
\quad
\begin{blockarray}{ccc}
\tilde{g}ddC & \tilde{g}\tilde{e}eC\\
\begin{block}{(cc)c}
 -\frac{1}{4} \sqrt{7-\sqrt{17}} & \frac{\sqrt{9+\sqrt{17}}}{4} & \tilde{g} g\tilde{e}eC\\
 \frac{\sqrt{9+\sqrt{17}}}{4} & \frac{\sqrt{7-\sqrt{17}}}{4} & \tilde{g}g\tilde{b}bC \\
    \end{block}
\end{blockarray}
$
}

\[
\begin{blockarray}{ccc}
f\tilde{c}_1c3 & f\tilde{c}_2c3 \\
\begin{block}{(cc)c}
 \frac{1}{8} \left(1+\sqrt{17}\right) & \frac{1}{4} \sqrt{\frac{1}{2}
   \left(23-\sqrt{17}\right)} & ff\tilde{c}_1c3  \\
 \frac{1}{4} \sqrt{\frac{1}{2} \left(23-\sqrt{17}\right)} & \frac{1}{8}
   \left(-1-\sqrt{17}\right)& ff\tilde{c}_2c3 \\
    \end{block}
\end{blockarray}
\]

\renewcommand{\arraystretch}{1.5} 
\resizebox{\linewidth}{!}{%
$
\begin{blockarray}{ccc}
bb\tilde{b}F & be\tilde{e}F \\
\begin{block}{(cc)c}
-\frac{1}{\sqrt{2} }&-\frac{1}{\sqrt{2} }&b\tilde{b}e\tilde{e}F\\
 \frac{1}{\sqrt{2} } &-\frac{1}{\sqrt{2} }& b\tilde{b}\tilde{e}eF\\ 
    \end{block}
\end{blockarray},
\quad
\begin{blockarray}{ccc}
\tilde{b}e\tilde{e}F &\tilde{b}\tilde{e}eF \\
\begin{block}{(cc)c}
\frac{1}{\sqrt{2} }&-\frac{1}{\sqrt{2} }& \tilde{b}bb\tilde{b}F\\
 \frac{1}{\sqrt{2} } &\frac{1}{\sqrt{2} }& \tilde{b}be\tilde{e}F \\ 
    \end{block}
\end{blockarray},
\quad
\begin{blockarray}{ccc}
ge\tilde{e}F & g\tilde{e}gF \\
\begin{block}{(cc)c}
-\frac{1}{\sqrt{2} }&-\frac{1}{\sqrt{2} }& g\tilde{g}\tilde{e}eF \\
 \frac{1}{\sqrt{2} } &-\frac{1}{\sqrt{2} }& g\tilde{g}\tilde{g}gF \\ 
    \end{block}
\end{blockarray},
\quad
\begin{blockarray}{ccc}
\tilde{g}\tilde{e}eF & \tilde{g}\tilde{g}gF  \\
\begin{block}{(cc)c}
-\frac{1}{\sqrt{2} }&\frac{1}{\sqrt{2} }& \tilde{g}ge\tilde{e}F\\
 -\frac{1}{\sqrt{2} } &-\frac{1}{\sqrt{2} }& \tilde{g}g\tilde{e}eF \\ 
    \end{block}
\end{blockarray},$
}

\renewcommand{\arraystretch}{1.5} 
\resizebox{\linewidth}{!}{%
$
\begin{blockarray}{ccc}
ff_1f6 & ff_2f6 \\
\begin{block}{(cc)c}
0&-1 & fff_1f6\\
 -1 & 0 & fff_2f6 \\ 
    \end{block}
\end{blockarray},
\quad
\begin{blockarray}{ccc}
dddC & d\tilde{e}eC \\
\begin{block}{(cc)c}
0&-1 & ddddC\\
 1 & 0 & dd\tilde{e}eC \\ 
    \end{block}
\end{blockarray},
\quad
\begin{blockarray}{ccc}
ddd\tilde{C} &de\tilde{e}\tilde{C} \\
\begin{block}{(cc)c}
0& 1 & dddd\tilde{C}\\
 -1 & 0 & dde\tilde{e}\tilde{C} \\     \end{block}
\end{blockarray}
$
}

\renewcommand{\arraystretch}{1.5} 
\resizebox{\linewidth}{!}{%
$
\begin{blockarray}{ccc}
bdd\tilde{C}&be\tilde{e}\tilde{C} \\
\begin{block}{(cc)c}
 -\frac{1}{4} \sqrt{7-\sqrt{17}} & \frac{\sqrt{9+\sqrt{17}}}{4} & b\tilde{b}e\tilde{e}\tilde{C}\\
 \frac{\sqrt{9+\sqrt{17}}}{4} & \frac{\sqrt{7-\sqrt{17}}}{4} & b\tilde{b}g\tilde{g}\tilde{C} \\    \end{block}
\end{blockarray},
\quad
\begin{blockarray}{ccc}
\tilde{b}e\tilde{e}\tilde{C} &\tilde{b}g\tilde{g}\tilde{C} \\
\begin{block}{(cc)c}
 \frac{\sqrt{7-\sqrt{17}}}{4} & -\frac{1}{4} \sqrt{9+\sqrt{17}} & \tilde{b}bdd\tilde{C}\\
 -\frac{1}{4} \sqrt{9+\sqrt{17}} & -\frac{1}{4} \sqrt{7-\sqrt{17}} & \tilde{b}be\tilde{e}\tilde{C}\\
    \end{block}
\end{blockarray}
$
}

\[
\begin{blockarray}{cccc}
eb\tilde{b}F &ee\tilde{e}F &e\tilde{e}eF \\
\begin{block}{(ccc)c}
 \frac{1}{2} \sqrt{-3+\sqrt{17}} & \frac{1}{2} \sqrt{\frac{1}{2}
   \left(5-\sqrt{17}\right)} & \frac{1}{4} \left(-1+\sqrt{17}\right) & e\tilde{e}e\tilde{e}F\\
 -\frac{1}{2} \sqrt{\frac{1}{2} \left(1+\sqrt{17}\right)} & \frac{1}{2} & \frac{1}{2}
   \sqrt{\frac{1}{2} \left(5-\sqrt{17}\right)}  & e\tilde{e}\tilde{e}eF\\
 \frac{1}{4} \left(-3+\sqrt{17}\right) & \frac{1}{2} \sqrt{\frac{1}{2}
   \left(1+\sqrt{17}\right)} & -\frac{1}{2} \sqrt{-3+\sqrt{17}}& e\tilde{e}\tilde{g}gF \\
    \end{block}
\end{blockarray}
\]

\[
\begin{blockarray}{cccc}
\tilde{e}e\tilde{e}F &\tilde{e}\tilde{e}eF &\tilde{e}\tilde{g}gF \\
\begin{block}{(ccc)c}
-\frac{1}{2} \sqrt{-3+\sqrt{17}} & \frac{1}{2} \sqrt{\frac{1}{2}
   \left(1+\sqrt{17}\right)} & \frac{1}{4} \left(3-\sqrt{17}\right) & \tilde{e}eb\tilde{b}F \\
 -\frac{1}{2} \sqrt{\frac{1}{2} \left(5-\sqrt{17}\right)} & -\frac{1}{2} & -\frac{1}{2}
   \sqrt{\frac{1}{2} \left(1+\sqrt{17}\right)}  & \tilde{e}ee\tilde{e}F\\
 \frac{1}{4} \left(1-\sqrt{17}\right) & -\frac{1}{2} \sqrt{\frac{1}{2}
   \left(5-\sqrt{17}\right)} & \frac{1}{2} \sqrt{-3+\sqrt{17}} & \tilde{e}e\tilde{e}eF\\
    \end{block}
\end{blockarray}
\]

\[
\begin{blockarray}{cccc}
eddC &e\tilde{e}eC &e\tilde{b}bC \\
\begin{block}{(ccc)c}
-\sqrt{\frac{1}{2} \left(5-\sqrt{17}\right)} & -\sqrt{-\frac{5}{16}+\frac{3
   \sqrt{17}}{16}} & -\sqrt{-\frac{19}{16}+\frac{5 \sqrt{17}}{16}}  &  e\tilde{e}ddC\\
 0 & -\frac{1}{4} \sqrt{7-\sqrt{17}} & \frac{\sqrt{9+\sqrt{17}}}{4}&e\tilde{e}\tilde{e}eC  \\
 -\sqrt{\frac{1}{2} \left(-3+\sqrt{17}\right)} & \frac{1}{2} \sqrt{\frac{1}{2}
   \left(7-\sqrt{17}\right)} & \frac{1}{4} \left(-3+\sqrt{17}\right)&e\tilde{e} \tilde{b}bC  \\
    \end{block}
\end{blockarray}
\]

\[
\begin{blockarray}{cccc}
\tilde{e}ddC &\tilde{e}\tilde{e}eC & \tilde{e}\tilde{b}bC \\
\begin{block}{(ccc)c}
 \sqrt{\frac{1}{2} \left(5-\sqrt{17}\right)} & 0 & \sqrt{\frac{1}{2}
   \left(-3+\sqrt{17}\right)}  & \tilde{e}eddC \\
 \sqrt{-\frac{5}{16}+\frac{3 \sqrt{17}}{16}} & \frac{\sqrt{7-\sqrt{17}}}{4} &
   -\frac{1}{2} \sqrt{\frac{1}{2} \left(7-\sqrt{17}\right)} &\tilde{e}e\tilde{e}eC \\
 \sqrt{-\frac{19}{16}+\frac{5 \sqrt{17}}{16}} & -\frac{1}{4} \sqrt{9+\sqrt{17}} &
   \frac{1}{4} \left(3-\sqrt{17}\right) & \tilde{e}e\tilde{b}bC\\
    \end{block}
\end{blockarray}
\]

\[
\begin{blockarray}{cccc}
cc_1\tilde{c}5 &cc_2\tilde{c}5 & cff5 \\
\begin{block}{(ccc)c}
 \frac{1}{2} \left(-3+\sqrt{17}\right) & \frac{1}{2} \sqrt{-3+\sqrt{17}} &
   -\sqrt{-\frac{19}{4}+\frac{5 \sqrt{17}}{4}} & c\tilde{c}c\tilde{c}5 \\
 -\sqrt{-\frac{19}{4}+\frac{5 \sqrt{17}}{4}} & \frac{1}{4} \sqrt{\frac{1}{2}
   \left(23-\sqrt{17}\right)} & \frac{1}{8} \left(13-3 \sqrt{17}\right) & c\tilde{c}f_1f5\\
 \frac{1}{2} \sqrt{-3+\sqrt{17}} & \frac{1}{8} \left(7-\sqrt{17}\right) & \frac{1}{4}
   \sqrt{\frac{1}{2} \left(23-\sqrt{17}\right)} & c\tilde{c}f_2f5\\
    \end{block}
\end{blockarray}
\]

\[
\begin{blockarray}{cccc}
\tilde{c}c\tilde{c}5 & \tilde{c}f_1f5 &  \tilde{c}f_2f5 \\
\begin{block}{(ccc)c}
 \frac{1}{2} \left(3-\sqrt{17}\right) & \sqrt{-\frac{19}{4}+\frac{5 \sqrt{17}}{4}} &
   -\frac{1}{2} \sqrt{-3+\sqrt{17}} & \tilde{c}cc_1\tilde{c}5\\
 -\frac{1}{2} \sqrt{-3+\sqrt{17}} & -\frac{1}{4} \sqrt{\frac{1}{2}
   \left(23-\sqrt{17}\right)} & \frac{1}{8} \left(-7+\sqrt{17}\right) & \tilde{c}cc_1\tilde{c}5\\
 \sqrt{-\frac{19}{4}+\frac{5 \sqrt{17}}{4}} & \frac{1}{8} \left(-13+3 \sqrt{17}\right)
   & -\frac{1}{4} \sqrt{\frac{1}{2} \left(23-\sqrt{17}\right)} & \tilde{c}cff5\\  
    \end{block}
\end{blockarray}
\]

\[
\begin{blockarray}{cccc}
fc\tilde{c}5 & ff_1f5 &  ff_2f5 \\
\begin{block}{(ccc)c}
 \frac{1}{2} \left(5-\sqrt{17}\right) & 0 & \sqrt{-\frac{19}{2}+\frac{5 \sqrt{17}}{2}}
 &ffc\tilde{c}5  \\
 0 & 1 & 0 &fff_1f5\\
 \sqrt{-\frac{19}{2}+\frac{5 \sqrt{17}}{2}} & 0 & \frac{1}{2} \left(-5+\sqrt{17}\right)
   &fff_2f5 \\
   \end{block}
\end{blockarray}
\]

\[
\begin{blockarray}{cccc}
edd\tilde{C} & ee\tilde{e}d\tilde{C} &  eg\tilde{g}\tilde{C}  \\
\begin{block}{(ccc)c}
\sqrt{\frac{1}{2} \left(5-\sqrt{17}\right)} & 0 & \sqrt{\frac{1}{2}
   \left(-3+\sqrt{17}\right)}& e\tilde{e}dd\tilde{C} \\
 -\sqrt{-\frac{5}{16}+\frac{3 \sqrt{17}}{16}} & -\frac{1}{4} \sqrt{7-\sqrt{17}} &
   \frac{1}{2} \sqrt{\frac{1}{2} \left(7-\sqrt{17}\right)} & e\tilde{e}e\tilde{e}\tilde{C}  \\
 -\sqrt{-\frac{19}{16}+\frac{5 \sqrt{17}}{16}} & \frac{\sqrt{9+\sqrt{17}}}{4} &
   \frac{1}{4} \left(-3+\sqrt{17}\right) & e\tilde{e}g\tilde{g}\tilde{C}\\
   \end{block}
\end{blockarray}
\]

\[
\begin{blockarray}{cccc}
f\tilde{c}_1c2 & f\tilde{c}_2c2 &  f\tilde{a}a2  \\
\begin{block}{(ccc)c}
\frac{1}{4} \left(-3+\sqrt{17}\right) & \frac{1}{2} \sqrt{\frac{1}{2}
   \left(7-\sqrt{17}\right)} & -\sqrt{\frac{1}{2} \left(-3+\sqrt{17}\right)} & ff\tilde{c}_1c2\\
 \frac{1}{2} \sqrt{\frac{1}{2} \left(7-\sqrt{17}\right)} & \frac{1}{2} & \frac{1}{2}
   \sqrt{\frac{1}{2} \left(-1+\sqrt{17}\right)}& ff\tilde{c}_2c2  \\
 -\sqrt{\frac{1}{2} \left(-3+\sqrt{17}\right)} & \frac{1}{2} \sqrt{\frac{1}{2}
   \left(-1+\sqrt{17}\right)} & \frac{1}{4} \left(5-\sqrt{17}\right) & ff\tilde{a}a2\\
   \end{block}
\end{blockarray}
\]

\[
\begin{blockarray}{cccc}
\tilde{e}dd\tilde{C} & \tilde{e}e\tilde{e}\tilde{C}&  \tilde{e}g\tilde{g}\tilde{C} \\
\begin{block}{(ccc)c}
 -\sqrt{\frac{1}{2} \left(5-\sqrt{17}\right)} & \sqrt{-\frac{5}{16}+\frac{3
   \sqrt{17}}{16}} & \sqrt{-\frac{19}{16}+\frac{5 \sqrt{17}}{16}} & \tilde{e}edd\tilde{C}\\
 0 & \frac{\sqrt{7-\sqrt{17}}}{4} & -\frac{1}{4} \sqrt{9+\sqrt{17}} & \tilde{e}ee\tilde{e}\tilde{C}\\
 -\sqrt{\frac{1}{2} \left(-3+\sqrt{17}\right)} & -\frac{1}{2} \sqrt{\frac{1}{2}
   \left(7-\sqrt{17}\right)} & \frac{1}{4} \left(3-\sqrt{17}\right)& \tilde{e}eg\tilde{g}\tilde{C} \\
   \end{block}
\end{blockarray}
\]

\renewcommand{\arraystretch}{1.5} 
\resizebox{\linewidth}{!}{%
$
\begin{blockarray}{ccccc}
db\tilde{b}F & de\tilde{e}F & d\tilde{e}eF & d\tilde{g}gF \\
\begin{block}{(cccc)c}
0 & \frac{1}{\sqrt{2}} & -\frac{1}{4} \sqrt{\frac{1}{2} \left(7-\sqrt{17}\right)} &
   \frac{1}{8} \left(-1-\sqrt{17}\right) & ddb\tilde{b}F \\
 -\frac{1}{\sqrt{2}} & 0 & \frac{1}{8} \left(-1-\sqrt{17}\right) & \frac{1}{4}
   \sqrt{\frac{1}{2} \left(7-\sqrt{17}\right)} & dde\tilde{e}F \\
 \frac{1}{4} \sqrt{\frac{1}{2} \left(7-\sqrt{17}\right)} & \frac{1}{8}
   \left(1+\sqrt{17}\right) & 0 & \frac{1}{\sqrt{2}} & dd\tilde{e}eF \\
 \frac{1}{8} \left(1+\sqrt{17}\right) & -\frac{1}{4} \sqrt{\frac{1}{2}
   \left(7-\sqrt{17}\right)} & -\frac{1}{\sqrt{2}} & 0 & dd\tilde{g}gF\\
   \end{block}
\end{blockarray}
$
}

\renewcommand{\arraystretch}{1.5} 
\resizebox{\linewidth}{!}{%
$
\begin{blockarray}{ccccc}
cc_1\tilde{c}4 &cc_1\tilde{c}4 &c \tilde{c}c4 &  cff4 \\
\begin{block}{(cccc)c}
 \frac{1}{8} \left(1-\sqrt{17}\right) & -\sqrt{\frac{5}{128}+\frac{3 \sqrt{17}}{128}} &
   \frac{1}{16} \left(9-\sqrt{17}\right) & \sqrt{\frac{19}{64}+\frac{5 \sqrt{17}}{64}}
  & c\tilde{c}c\tilde{c}4 \\
 \frac{1}{8} \left(9-\sqrt{17}\right) & \sqrt{-\frac{3}{128}+\frac{11 \sqrt{17}}{128}}
   & \frac{1}{16} \left(-11+3 \sqrt{17}\right) & \sqrt{-\frac{101}{64}+\frac{29
   \sqrt{17}}{64}} & c\tilde{c}\tilde{c}c4\\
 -\frac{1}{4} \sqrt{-1+\sqrt{17}} & \frac{3 \sqrt{5-\sqrt{17}}}{8} &
   -\sqrt{\frac{19}{64}+\frac{5 \sqrt{17}}{64}} & \frac{1}{4}  & c\tilde{c}f_1f4 \\
 -\frac{1}{2} \sqrt{-3+\sqrt{17}} & \frac{1}{8} \left(1+\sqrt{17}\right) & \frac{1}{2}
   \sqrt{-3+\sqrt{17}} & -\frac{1}{4} \sqrt{\frac{1}{2} \left(71-17 \sqrt{17}\right)}
   & c\tilde{c}f_2f4 \\
   \end{block}
\end{blockarray}
$
}

\renewcommand{\arraystretch}{1.5} 
\resizebox{\linewidth}{!}{%
$
\begin{blockarray}{ccccc}
\tilde{c}c\tilde{c}4 & \tilde{c}\tilde{c}c4 & \tilde{c}f_1f4 &  \tilde{c}f_1f4 \\
\begin{block}{(cccc)c}
 \frac{1}{8} \left(-1+\sqrt{17}\right) & \frac{1}{8} \left(-9+\sqrt{17}\right) &
   \frac{1}{4} \sqrt{-1+\sqrt{17}} & \frac{1}{2} \sqrt{-3+\sqrt{17}} & \tilde{c} cc_1\tilde{c}4\\
 \sqrt{\frac{5}{128}+\frac{3 \sqrt{17}}{128}} & -\sqrt{-\frac{3}{128}+\frac{11
   \sqrt{17}}{128}} & -\frac{3}{8} \sqrt{5-\sqrt{17}} & \frac{1}{8}
   \left(-1-\sqrt{17}\right) &  \tilde{c} cc_2\tilde{c}4 \\
 \frac{1}{16} \left(-9+\sqrt{17}\right) & \frac{1}{16} \left(11-3 \sqrt{17}\right) &
   \sqrt{\frac{19}{64}+\frac{5 \sqrt{17}}{64}} & -\frac{1}{2} \sqrt{-3+\sqrt{17}} & \tilde{c} c\tilde{c}c4\\
 -\sqrt{\frac{19}{64}+\frac{5 \sqrt{17}}{64}} & -\sqrt{-\frac{101}{64}+\frac{29
   \sqrt{17}}{64}} & -\frac{1}{4} & \frac{1}{4} \sqrt{\frac{1}{2} \left(71-17
   \sqrt{17}\right)} &  \tilde{c} c ff 4 \\
   \end{block}
\end{blockarray}
$
}

\renewcommand{\arraystretch}{1.5} 
\resizebox{\linewidth}{!}{%
$
\begin{blockarray}{cccccc}
fc\tilde{c}4& f\tilde{c}_1c4 & f\tilde{c}2c4 & ff_1f4 & ff_2f4  \\
\begin{block}{(ccccc)c}
 \frac{1}{16} \left(-9+\sqrt{17}\right) & \frac{1}{16} \left(-7-\sqrt{17}\right) &
   \frac{\sqrt{5-\sqrt{17}}}{2} & \sqrt{-\frac{101}{128}+\frac{29 \sqrt{17}}{128}} &
   \sqrt{-\frac{13}{128}+\frac{5 \sqrt{17}}{128}}  & ffc \tilde{c}4 \\
 \frac{1}{16} \left(-7-\sqrt{17}\right) & \frac{1}{16} \left(-9+\sqrt{17}\right) &
   -\frac{1}{2} \sqrt{5-\sqrt{17}} & -\sqrt{-\frac{101}{128}+\frac{29 \sqrt{17}}{128}}
   & -\sqrt{-\frac{13}{128}+\frac{5 \sqrt{17}}{128}} & ff\tilde{c}_1c4\\
 \frac{\sqrt{5-\sqrt{17}}}{2} & -\frac{1}{2} \sqrt{5-\sqrt{17}} & 0 & 0 &
   -\sqrt{\frac{1}{2} \left(-3+\sqrt{17}\right)} &  ff\tilde{c}_2c4 \\
 \sqrt{-\frac{101}{128}+\frac{29 \sqrt{17}}{128}} & -\sqrt{-\frac{101}{128}+\frac{29
   \sqrt{17}}{128}} & 0 & \frac{1}{16} \left(-7-\sqrt{17}\right) & \frac{1}{16}
   \left(-13+5 \sqrt{17}\right) & fff_1f4\\
 \sqrt{-\frac{13}{128}+\frac{5 \sqrt{17}}{128}} & -\sqrt{-\frac{13}{128}+\frac{5
   \sqrt{17}}{128}} & -\sqrt{\frac{1}{2} \left(-3+\sqrt{17}\right)} & \frac{1}{16}
   \left(-13+5 \sqrt{17}\right) & \frac{1}{16} \left(9-\sqrt{17}\right) & fff_2f4 \\
   \end{block}
\end{blockarray}
 $
}

\newcommand{\urlprefix}{}
\bibliographystyle{alpha}
\bibliography{bibliography}

\end{document}